\documentclass{amsart}

\usepackage{amsmath,amssymb,amsfonts}
\usepackage{esint, cancel}
\usepackage{amsthm}
\usepackage{mathrsfs}
\usepackage{geometry}
\usepackage{verbatim}
\usepackage{graphicx}
\usepackage{color}

\usepackage[colorlinks=true,urlcolor=blue, citecolor=red,linkcolor=blue,
linktocpage,pdfpagelabels, bookmarksnumbered,bookmarksopen]{hyperref}
\usepackage[hyperpageref]{backref}

\newtheorem{theorem}{Theorem}[section]
\newtheorem{lemma}[theorem]{Lemma}
\newtheorem{proposition}[theorem]{Proposition}
\newtheorem{corollary}[theorem]{Corollary}
\newtheorem*{main}{Main Theorem}
\newtheorem*{coro1}{Corollary 1}
\newtheorem*{coro2}{Corollary 2}
\theoremstyle{definition}
\newtheorem{remark}[theorem]{Remark}
\newtheorem{example}[theorem]{Example}
\newtheorem{definition}[theorem]{Definition}

\newtheorem*{ack}{Acknowledgments}

\numberwithin{equation}{section}

\author[Bozzola]{Francesco Bozzola}
\address[F.\ Bozzola]{Dipartimento di Scienze Matematiche, Fisiche e Informatiche
	\newline\indent
	Universit\`a di Parma
	\newline\indent
	Parco Area delle Scienze 53/a, Campus, 43124 Parma, Italy}
\email{francesco.bozzola@unipr.it}

\author[Brasco]{Lorenzo Brasco}
\address[L.\ Brasco]{Dipartimento di Matematica e Informatica
	\newline\indent
	Universit\`a degli Studi di Ferrara
	\newline\indent
	Via Machiavelli 35, 44121 Ferrara, Italy}
\email{lorenzo.brasco@unife.it}

\dedicatory{In loving memory of Silver Sirotti, 50 years after his death}

\date{\today}


\title{Capacitary inradius and Poincar\'e-Sobolev inequalities}

\begin{document}

\begin{abstract}
We prove a two-sided estimate on the sharp $L^p$ Poincar\'e constant of a general open set, in terms of a capacitary variant of its inradius.
This extends a result by Maz'ya and Shubin, originally devised for the case $p=2$, in the subconformal regime.
We cover the whole range of $p$, by allowing in particular the extremal cases $p=1$ (Cheeger's constant) and $p=N$ (conformal case), as well. We also discuss the more general case of the sharp Poincar\'e-Sobolev embedding constants and get an analogous result. Finally, we present a brief discussion on the superconformal case, as well as some examples and counter-examples.
\end{abstract}

\maketitle

\begin{center}
\begin{minipage}{10cm}
\small
\tableofcontents
\end{minipage}
\end{center}

\section{Introduction}

\subsection{Balls VS. Poincar\'e}
To start with, let us introduce the following two quantities
\[
\lambda_p(\Omega)=\inf_{\varphi\in C^\infty_0(\Omega)} \left\{\int_\Omega |\nabla \varphi|^p\,dx\, :\, \|\varphi\|_{L^p(\Omega)}=1\right\},
\]
and
\[
r_\Omega=\sup\Big\{r>0\, :\, \text{there exists a ball}\ B_r(x_0)\subseteq\Omega\Big\},
\]
which are associated to an open set $\Omega\subseteq\mathbb{R}^N$. Here the exponent $p$ is between $1$ and $\infty$, while the dimension $N$ will always be larger than or equal to $2$. The symbol $C^\infty_0(\Omega)$ indicates the space of infinitely differentiable functions, whose support is a compact subset of $\Omega$. 
\par
The constant $\lambda_p(\Omega)$ is the sharp $L^p$ Poincar\'e constant for functions ``vanishing at the boundary'' of $\Omega$. It may happen that this constant is zero: accordingly, the set $\Omega$ does not support such an inequality. This occurs for example whenever $\Omega$ contains balls of arbitrarily large radius. This fact can be made ``quantitative'' through the following upper bound
\begin{equation}
\label{1intro}
\lambda_p(\Omega)\le \lambda_p(B_1)\,\left(\frac{1}{r_\Omega}\right)^p.
\end{equation}
where $B_1=\{x\in\mathbb{R}^N\, :\, |x|<1\}$.
We thus obtain a first connection between these two quantities.
We also know that for $p>N$, it is possible to prove the lower bound
\begin{equation}
\label{2intro}
\lambda_p(\Omega)\ge C_{N,p}\,\left(\frac{1}{r_\Omega}\right)^p,
\end{equation}
as well.
This is again valid for {\it every} open set $\Omega\subseteq\mathbb{R}^N$, see \cite[Theorem 1.4.1]{Po1}, \cite[Theorem 1.1]{Vit} and more recently \cite[Theorem 4.5]{BozBra} and \cite[Theorem 5.4]{BraPriZag2}. 
\par
On the contrary, when $1\le p\le N$ it is no more possible to bound $\lambda_p(\Omega)$ from below in terms of $r_\Omega$. There is a problem of ``removability'' in this case. In other words,
the geometric object $r_\Omega$ is affected by the removal of single points, while the latter are ``invisible'' sets for $\lambda_p(\Omega)$, in the range $1\le p\le N$. The typical counterexample to the lower bound is then given by $\mathbb{R}^N\setminus \mathbb{Z}^N$. Imposing to functions to vanish in an arbitrarily small neighborhood of the points of a lattice is not enough to get a $L^p$ Poincar\'e inequality, when $1\le p\le N$. 
\par
In more precise words, in general the quantity $\lambda_p(\Omega)$ is not affected by the removal of compact subsets $\Sigma\subseteq\Omega$ such that their {\it $p-$capacity relative to a ball $B_R(x_0)$} 
\[
\mathrm{cap}_p(\Sigma;B_R(x_0))=\inf_{\varphi\in C^\infty_0(B_R(x_0))}\left\{\int_{B_R(x_0)} |\nabla \varphi|^p\,dx\, :\, \varphi \ge 1 \text{ on } \Sigma\right\},\qquad \Sigma\Subset B_R(x_0),
\]
is zero (see for example \cite[Chapter 2]{Flu} or \cite[Chapter 2]{Maz} for more details on $p-$capacity).
\vskip.2cm\noindent
To restore the situation, without imposing any additional condition on the open sets, a natural idea would be that of replacing the {\it inradius} $r_\Omega$ with a ``relaxed'' version, which allows the balls to be contained in $\Omega$ only in a ``capacitary sense''. Thus, a first naive attempt would be that of replacing the usual inradius $r_\Omega$ with the following capacitary variant
\begin{equation}
\label{inradiocap}
\mathfrak{R}_{\Omega}:=\sup\Big\{r>0\, :\, \exists x_0\in\mathbb{R}^N\ \text{such that}\ \mathrm{cap}_p\left(\overline{B_r(x_0)}\setminus\Omega;B_{2r}(x_0)\right)=0\Big\}.
\end{equation}
Observe that it is necessary to use the concept of {\it relative} capacity, in order to include in the discussion the conformal case $p=N$, as well. Indeed, we recall that for every compact set $\Sigma\subseteq\mathbb{R}^N$ its {\it absolute $N-$capacity}, defined by 
\[
\mathrm{cap}_N(\Sigma):=\inf_{\varphi\in C^\infty_0(\mathbb{R}^N)}\left\{\int_{\mathbb{R}^N} |\nabla \varphi|^N\,dx\,:\, \varphi\ge 1\ \text{on}\ \Sigma\right\},
\]
is always zero,
due to the scale invariance of the $N-$Dirichlet integral (see \cite[pages 148--149]{Maz}). We point out that the will of including the case $p=N$ in the discussion is not a mere technicality: indeed, for $p=2$ the quantity $\lambda_p(\Omega)$ coincides with the bottom of the spectrum of the Dirichlet-Laplacian on $\Omega$ (see \cite[Chapter 10, Section 1.1]{BirSol}). Thus, the case $p=N=2$ is important in Spectral Theory.
\vskip.2cm\noindent
However, even by using the inradius defined by \eqref{inradiocap}, one could show that a lower bound on $\lambda_p(\Omega)$ is not possible, without any further assumption on the open set $\Omega$. We refer to Example \ref{exa:gamma0} below for a counter-example. The main problem in the definition \eqref{inradiocap} is the lack of some ``uniformity'' in the portion of complement $\mathbb{R}^N\setminus\Omega$ that this capacitary variant of the inradius can detect.
\par
In order to circumvent this problem, in \cite{MS} Maz'ya and Shubin proposed to work with the concept of {\it negligible set} (in the sense of Molchanov), for a fixed parameter $0<\gamma<1$ (see also \cite[Chapter 18, Section 7]{Maz}). This leads us to the following
\begin{definition}
Let $1\le p<\infty$ and $0<\gamma<1$, we say that a compact set $\Sigma\subseteq \overline{B_r(x_0)}$ 
is {\it $(p,\gamma)-$negligible} if 
\[
\mathrm{cap}_p(\Sigma;B_{2r}(x_0))\le \gamma\,\mathrm{cap}_p\left(\overline{B_r(x_0)};B_{2r}(x_0)\right).
\]
Accordingly, we consider the {\it capacitary inradius of $\Omega$}, defined as follows\footnote{
For ease of simplicity, we prefer to simply call it {\it capacitary inradius}, rather than {\it $(p,\gamma)-$capacitary inradius} or something similar.}
\begin{equation}
\label{inradiocapmod}
R_{p,\gamma}(\Omega):=\sup\Big\{r>0\, :\, \exists x_0\in\mathbb{R}^N\ \text{such that}\ \overline{B_r(x_0)}\setminus\Omega\ \text{is $(p,\gamma)-$negligible}\Big\}.
\end{equation}
see Figure \ref{fig:capin}. From its definition, we can immediately record the following two properties
\[
r_\Omega\le R_{p,\gamma}(\Omega),\ \text{for every}\ 0<\gamma<1,\qquad\text{and}\qquad \gamma \mapsto R_{p,\gamma}(\Omega) \ \text{is monotone non-decreasing}.
\]
\end{definition}
\begin{figure}
\includegraphics[scale=.25]{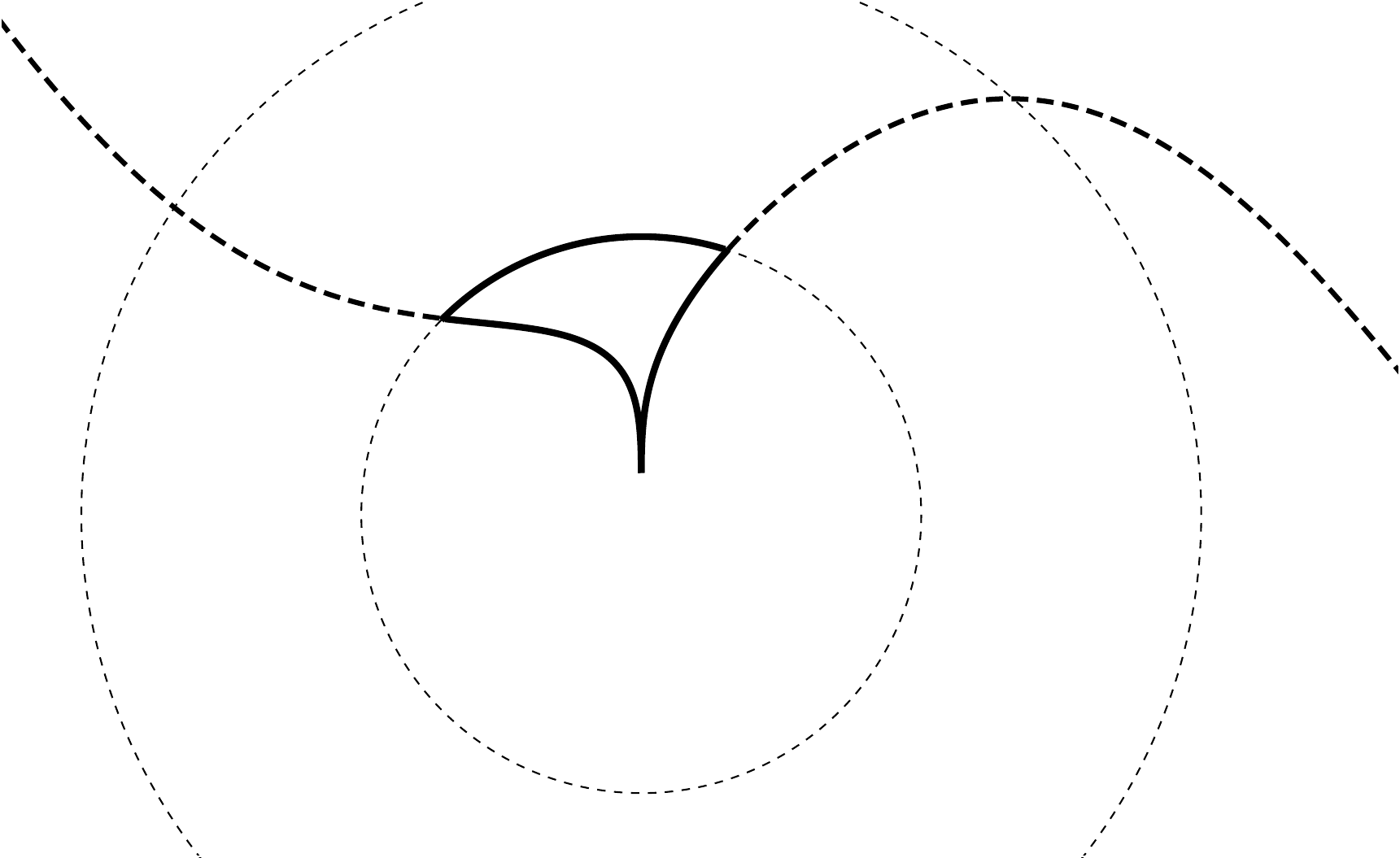}
\caption{Contoured by the bold line, the set $\overline{B_r(x_0)}\setminus\Omega$. For $\gamma$ small enough, its $p-$capacity is smaller than $\gamma$ times the capacity of the whole ball (the smaller one, in dashed line). Accordingly, this radius $r$ is a feasible competitor in the definition of the capacitary inradius. The largest ball in dashed line corresponds to the ``box'' $B_{2r}(x_0)$ which is used to compute the relative capacity.}
\label{fig:capin}
\end{figure}

\begin{remark}
The analysis of the paper \cite{MS} was confined to the case $p=2$. Moreover, the definition of capacitary inradius there contained is slightly different from ours \eqref{inradiocapmod}, since the authors use the absolute capacity
\[
\mathrm{cap}_2(\Sigma):=\inf_{\varphi\in C^\infty_0(\mathbb{R}^N)}\left\{\int_{\mathbb{R}^N} |\nabla \varphi|^2\,dx\,:\, \varphi\ge 1\ \text{on}\ \Sigma\right\}.
\]
For this reason, the case $p=N=2$ is not explicitly treated in \cite{MS}, in light of the discussion above. We will come back on a comparison between our result and those of \cite{MS} in a while.
\end{remark}
\subsection{Results of the paper}
The following two-sided estimate is the main result of the paper. This can be seen as an extension of \cite[Theorem 1.1]{MS}, to the case $p\not=2$. We refer to Remark \ref{rem:constant} below for a comment on the constant $C_{N,p,\gamma}$.
\begin{main}
Let $1 \leq p \leq N$, $0<\gamma<1$ and let $\Omega\subseteq\mathbb{R}^N$ be an open set. Then we have 
\[
\sigma_{N,p}\,\gamma\,\left(\frac{1}{R_{p,\gamma}(\Omega)}\right)^p\le \lambda_p(\Omega)\le C_{N,p,\gamma}\,\left(\frac{1}{R_{p,\gamma}(\Omega)}\right)^p,
\]
with the constant $C_{N,p,\gamma}$ which diverges to $+\infty$, as $\gamma$ goes to $1$. In particular, we have
\[
\lambda_p(\Omega)>0\qquad \Longleftrightarrow\qquad R_{p,\gamma}(\Omega)<+\infty,
\]
and the last condition does not depend on $0<\gamma<1$.
\end{main}
As in \cite{MS}, the proof of this result is constructive and thus the constants $\sigma_{N,p}$ and $C_{N,p,\gamma}$ are computable, in principle. However, since they are very likely not sharp and their explicit expression is not particularly pleasant, we prefer to avoid writing them in the statement above.
\vskip.2cm\noindent
Before going further, we wish to highlight a couple of consequences of our main result: the first one is a simple rewriting of the statement, in the case $p=1$.
Indeed, we recall that for $p=1$, the quantity $\lambda_p(\Omega)$ actually coincides with the so-called {\it Cheeger constant} of $\Omega$, defined by
\[
h(\Omega)=\inf\left\{\frac{\mathcal{H}^{N-1}(\partial E)}{|E|}\, :\, E\Subset \Omega\ \text{open set with smooth boundary}\right\},
\]
see for example \cite[Theorem 2.3.1]{Maz} for a proof of this fact. We refer to \cite{Leo} and \cite{Parini} for an introduction to the Cheeger constant and the interesting problems connected with it. 
\par
We get the following two-sided estimate, which deserves to be explicitly written. 
\begin{coro1}
Let $0<\gamma<1$ and let $\Omega\subseteq\mathbb{R}^N$ be an open set. Then we have 
\[
\sigma_{N,1}\,\gamma\,\frac{1}{R_{1,\gamma}(\Omega)}\le h(\Omega)\le C_{N,1,\gamma}\,\frac{1}{R_{1,\gamma}(\Omega)},
\]
with the constant $C_{N,1,\gamma}$ which diverges to $+\infty$, as $\gamma$ goes to $1$. In particular, we have
\[
h(\Omega)>0\qquad \Longleftrightarrow\qquad R_{1,\gamma}(\Omega)<+\infty,
\]
and the last condition does not depend on $0<\gamma<1$.
\end{coro1}
A second consequence concerns the so-called {\it $p-$torsion function} of an open set $\Omega$. This function, denoted by $w_\Omega$, is informally defined as the solution of 
\[
-\Delta_p w_\Omega=1,\qquad \text{in}\ \Omega,
\] 
with homogeneous Dirichlet boundary conditions on $\partial\Omega$. For the precise definition in the case of a general open set, we refer to \cite[Section 2]{BR}, for example. The importance of this function in the context of the theory of Sobolev spaces is encoded in the following equivalence
\[
\lambda_p(\Omega)>0\qquad \Longleftrightarrow\qquad w_{\Omega}\in L^\infty(\Omega).
\]
Actually, this equivalence can be made ``quantitative''. Indeed, from \cite[Theorem 1.3]{BR}  and \cite[Theorem 9]{vdBB}, we know that
\begin{equation}
\label{lambdaw}
1\le \lambda_p(\Omega)\,\|w_{\Omega}\|_{L^\infty(\Omega)}^{p-1}\le \mathbf{D}_{N,p}.
\end{equation}
We also refer to \cite[Proposition 6]{BE} and \cite[Lemma 4.1]{GS} for the leftmost estimate, in the case of smooth bounded domains.
\par
By joining this two-sided estimate with that of the Main Theorem, we get the following
\begin{coro2}
Let $1< p \leq N$, $0<\gamma<1$ and let $\Omega\subseteq\mathbb{R}^N$ be an open set. Then we have
\[
\left(\frac{1}{C_{N,p,\gamma}}\right)^\frac{1}{p-1}\,\Big(R_{p,\gamma}(\Omega)\Big)^\frac{p}{p-1}\le \|w_{\Omega}\|_{L^\infty(\Omega)}\le \left(\frac{\mathbf{D}_{N,p}}{\gamma\,\sigma_{N,p}}\right)^\frac{1}{p-1}\,\Big(R_{p,\gamma}(\Omega)\Big)^\frac{p}{p-1},
\]
where $\sigma_{N,p}$ and $C_{N,p,\gamma}$ are the same constants as in the Main Theorem.
\end{coro2}
\begin{remark}
For completeness, let us discuss the counterpart of the previous result, with the classical inradius $r_\Omega$ in place of $R_{p,\gamma}(\Omega)$. The lower bound holds for every open set.
Indeed, by the comparison principle for the $p-$Laplacian, for every ball $B_r(x_0)\subseteq\Omega$ we have
\[
w_\Omega(x)\ge w_{B_{r}(x_0)}(x)=\frac{p-1}{p}\,\frac{1}{N^\frac{1}{p-1}}\,\left(r^\frac{p}{p-1}-|x-x_0|^\frac{p}{p-1}\right)_+.
\]
By passing to the essential supremum and using the arbitrariness of the ball, we obtain the sharp lower bound
\[
\|w_\Omega\|_{L^\infty(\Omega)}\ge \frac{p-1}{p}\,\frac{1}{N^\frac{1}{p-1}}\,\Big(r_\Omega\Big)^\frac{p}{p-1}.
\]
The upper bound on the contrary is not always true, unless some restrictions are imposed on the open sets. Once again, removability issues can be held responsible for the failure. It is known to be true for {\it convex sets} and for {\it planar multiply connected sets}, for example. In the first case, this is contained in \cite[Theorem 1.2]{DPGG} (see also \cite[Corollary 5.3]{BraPriZag}). In the second case, it can be obtained by combining the rightmost inequality in \eqref{lambdaw}, with the lower bound on $\lambda_p(\Omega)$ given by \cite[Theorem 3.4]{BozBra}. The special case $p=2$ for an open simply connected subset of the plane was contained in \cite[Corollary 1, equation (0.6)]{BC}.
\end{remark}

\subsection{Some comments on the main result}
We fairly admit that the present paper is very much inspired to \cite{MS}. Indeed, it was our original intention to expand the analysis of \cite{MS}, shed some light on the methods therein used and extend the results to the general case of the $L^p$ Poincar\'e inequality (and more generally to $L^q-L^p$ Poincar\'e-Sobolev inequalities, see Section \ref{sec:6}). 
\par
We remark at first that a two-sided estimate like that of our Main Theorem, still valid for every $1\le p\le N$, was already contained in the old version of Maz'ya's book \cite{Maz85}: with a brave and careful inspection, one could trace it back to \cite[Theorem 11.4.1]{Maz85} (this is \cite[Theorem 15.4.1]{Maz} in the new version). To be more precise, the latter is concerned with a slight variant of the capacitary inradius $R_{p,\gamma}(\Omega)$ introduced above, defined by replacing balls with cubes\footnote{In the notation and terminology of \cite[Theorem 11.4.1]{Maz85} and \cite[Theorem 15.4.1]{Maz}, this is the quantity $D_{p,l}(\Omega)$ with $l=1$, called {\it $(p,l)-$inner diameter} (see \cite[Definition 10.2.2]{Maz85} or \cite[Definition 14.2.2]{Maz}, by taking $Q_d=\mathbb{R}^n$, with the notation there). In the aforementioned result, the author proved that
\[
D_{p,1}(\Omega)\lesssim C\lesssim D_{p,1}(\Omega),
\]
where the constant $C$ in \cite{Maz, Maz85} coincides with $(\lambda_p(\Omega))^{-1/p}$, in our notation.}. 
Apart for the use of cubes in place of balls, the key point which marks the big difference with both \cite[Theorem 1.1]{MS} and our result, is that \cite[Theorem 11.4.1]{Maz85} is proved under a {\it restriction on the negligibility parameter} $\gamma$. In other words, for the arguments used in \cite{Maz, Maz85} it is needed that
\[
0<\gamma\le \gamma_{N,p}<1, 
\]
with $\gamma_{N,p}$ explicit and exponentially decaying to $0$, as $N$ goes to $\infty$ (see \cite[equation (10.1.2)]{Maz85} or \cite[equation (14.1.2)]{Maz}).
\par 
Maz'ya and Shubin in their paper \cite{MS} dropped this restriction, at least in the quadratic case $p=2$. Our main result then permits to overcome this limitation on $\gamma$ {\it for the whole range of $p$}, as well. Moreover, at the same price, we can get the same type of two-sided estimate for the quantities 
\[
\lambda_{p,q}(\Omega):= \inf_{\varphi\in C^\infty_0(\Omega)} \left\{\int_\Omega |\nabla u|^p\,dx\, :\, \|u\|^p_{L^q(\Omega)}=1\right\},
\]
for every subcritical exponent $q>p$. 
\par
 Indeed, the main interest of both \cite[Theorem 1.1]{MS} and our Main Theorem lies in the fact that the results hold {\it for every} $0<\gamma<1$. This is quite remarkable, since as $\gamma$ gets closer and closer to $1$, a ball which is $(p,\gamma)-$negligible is admitted to catch more and more portion of the complement of $\Omega$. This means that $R_{p,\gamma}(\Omega)$ starts to keep less and less memory of $\Omega$: nevertheless, as far as $\gamma<1$, it carries an information which is still enough to assure the validity of the $L^p$ Poincar\'e inequality (and even of the $L^q-L^p$ Poincar\'e-Sobolev inequalities).
\vskip.2cm\noindent
Even if we follow quite closely \cite{MS}, this does not mean that the proof of the Main Theorem is just a straightforward transposition of that of \cite[Theorem 1.1]{MS}. For example, in the proof of the upper bound, Maz'ya and Shubin rely very much on the representation formula for the {\it capacitary potential}, i.e. the function attaining the minimum value
\[
\mathrm{cap}_2\left(\overline{B_r(x_0)}\setminus\Omega;B_{2r}(x_0)\right).
\]
Such a potential can be expressed in terms of the fundamental solution of the Laplacian (more precisely, in terms of the Green function, at least in our case which uses the relative capacity). It is probably superfluous to mention that this is not possible for $p\not=2$, due to the nonlinearity of the relevant equation. Whenever possible, we also tried to simplify certain technical points of the original paper and add some explanations. 
 \par
We now wish to make some comments on the proofs.
\begin{itemize}
\item {\it Lower bound}: we proceed quite similarly to Maz'ya and Shubin. The key point is the use of a Maz'ya-Poincar\'e inequality for functions in a cube or a ball, vanishing on a Dirichlet region with positive capacity (the prototype of this type of results is \cite[Theorem 14.1.2]{Maz}).
We partially amend this strategy, by using a variant of such an inequality for functions on cubes, but with the capacity of the Dirichlet region computed with respect to a ball (this is a sort of ``mixed'' strategy, taken from our recent paper \cite{BozBra}). This permits to get the result by a {\it tiling argument} with cubes, rather than by a {\it covering argument} with balls as in \cite{MS}. This simplifies the argument, to a certain extent (it is not necessary to take into account the dimensional-dependent multiplicity factor of the covering). This gives a constant which is quantitatively rougher than that of \cite{MS}, but it is qualitatively comparable in terms of $\gamma$, i.e. our lower bound still decays to $0$ linearly with $\gamma$, when this goes to $0$ (compare it with \cite[equation (3.19)]{MS}).
\vskip.2cm\noindent
\item {\it Upper bound}: this is the point that requires greater care, in order to allow the parameter $\gamma$ to be arbitrarily close to $1$. Here as well we follow Maz'ya and Shubin, but as remarked above a nonlinear approach is now needed to get (or to judiciously estimate) the sharp constant in a subtle $L^1-L^p$ Poincar\'e-type inequality, for $p\not=2$. Even if we are not able to get the explicit expression for the extremals of this inequality, by suitably using some integral identities we can determine the optimal constant. The expression of such a constant is a bit involved and difficult to handle: nevertheless, by using a monotonicity property of the relative $p-$capacity of balls, which can be seen as a weaker version of {\it Gr\"otzsch's lemma} (see Lemma \ref{lm:cazzatella} and Remark \ref{rem:cazzatella}), we can finally get an estimate of the sharp constant which is handy and good enough for our purposes. All this part is the content of Section \ref{sec:3}.
\par
At a technical level, we also avoid the delicate approximation argument used in \cite{MS}, to replace $\overline{B_r(x_0)}\setminus\Omega$ (which may be very rough) with a smoother set. This is needed in \cite{MS} so to work with a capacitary potential which is sufficiently smooth and exploit the fact that this is harmonic. 
Here, on the contrary, we work directly with $\overline{B_r(x_0)}\setminus \Omega$ and show that, in place of a capacitary potential of this set, it is sufficient to take any ``almost'' minimizer of the relative $p-$capacity (and by density, this can be taken as smooth as we wish). Thus, we can be dispensed with the use of the PDE and simply use the minimality (or almost minimality) property of the function. This simplifies the argument, at the price of a slight quantitative worsening of the constant. This is not a big deal, since in any case the constants involved in the two-sided estimate are not sharp, both in \cite{MS} and in our paper.
On the contrary, at a qualitative level, our final estimate in terms of $\gamma$ is as good as that of Maz'ya and Shubin (see Remark \ref{rem:constant} and compare with \cite[equation (4.16)]{MS}).
\end{itemize}

\subsection{Plan of the paper}
The main notation and basic definitions are settled in Section \ref{sec:2}. In particular, we recall some basic properties of the capacity of balls and prove a ``quantified'' monotonicity property of this capacity, as a function of the radius (see Lemma \ref{lm:cazzatella}). In Section \ref{sec:3}, we obtain an explicit sharp $L^1-L^p$  Poincar\'e-type constant on balls (see Proposition \ref{coro:sharp1p}): as previously explained, this will be a key ingredient in the proof of the upper bound of the Main Theorem. Incidentally, as a byproduct of our analysis, we obtain the explicit value of a curious Cheeger-type constant for a ball (see Corollary \ref{coro:sharp11} and Remark \ref{oss:cheegertype}). 
We then enter into the core of the paper with Section \ref{sec:4}: by using a tiling argument in combination with a Maz'ya-Poincar\'e--type inequality, we derive the lower bound of our Main Theorem. 
The proof of the upper bound can be found in the subsequent Section \ref{sec:5}.
Particular attention is given to the quality of the constants involved in both the estimates. Section \ref{sec:6} is devoted to prove the extension of the Main Theorem to the case of general Poincar\'e-Sobolev embedding constants. In Section \ref{sec:7} we discuss the superconformal case $p>N$ and exhibit a sharp condition on $\gamma$ for having $R_{p,\gamma}(\Omega)=r_\Omega$. The paper is complemented by Appendix \ref{sec:A}, containing some counter-examples, aimed at clarifying some subtle aspects of the limit cases $\gamma=0$ and $\gamma\nearrow 1$, as well as to complete the discussion of Section \ref{sec:6} for the case $q<p$.

\begin{ack}
We thank Francesca Prinari for some useful conversations. The paper has been finalized during a one month staying of F.B. at the Laboratoire de Math\'ematiques of the Universit\'e Savoie Mont Blanc. The hosting institution and its facilities are kindly acknowledged. 
\par
F.\,B. is a member of the {\it Gruppo Nazionale per l'Analisi Matematica, la Probabilit\`a
e le loro Applicazioni} (GNAMPA) of the Istituto Nazionale di Alta Matematica (INdAM). F.\,B. has been financially supported by the joint Ph.D. program of the Universities of Ferrara, Modena \& Reggio Emilia and Parma. L.\,B. has been financially supported by the {\it Fondo di Ateneo per la Ricerca} {\sc FAR 2021} and the {\it Fondo per l'Incentivazione alla Ricerca Dipartimentale} {\sc FIRD 2022} of the University of Ferrara. 
\end{ack}

\section{Preliminaries}
\label{sec:2}

\subsection{Notation and basic definitions}

We will use the usual standard notations for $N-$dimensional balls and hypercubes, that is
\[
B_R(x_0)=\Big\{x\in\mathbb{R}^N\, :\, |x-x_0|<R\Big\},\qquad \text{for}\ x_0\in\mathbb{R}^N,\ R>0,
\]
and
\[
Q_R(x_0)=\prod_{i=1}^N (x_0^i-R,x_0^i+R),\qquad \text{for}\ x_0=(x_0^1,\dots,x_0^N)\in\mathbb{R}^N,\ R>0.
\]
When the center $x_0$ coincides with the origin, we will simply write $B_R$ and $Q_R$, respectively.
\par
For $1\le p<\infty$ and for an open set $E\subseteq\mathbb{R}^N$, we will denote by $W^{1,p}(E)$ 
the standard Sobolev space
\[
W^{1,p}(E)=\Big\{u\in L^p(E)\, :\, \nabla u\in L^p(E;\mathbb{R}^N)\Big\},
\]
endowed with its natural norm. The symbol $W^{1,p}_0(\Omega)$ will denote the closure of $C^\infty_0(\Omega)$ in $W^{1,p}(\Omega)$.
\begin{definition}
\label{defi:capacity}
Let $1\le p<\infty$, for every $E\subseteq\mathbb{R}^N$ open set and every $\Sigma\subseteq E$ compact set, we define the {\it $p-$capacity of $\Sigma$ relative to $E$} through the following minimization problem
\[
\mathrm{cap}_p(\Sigma;E)=\inf_{\varphi\in C^\infty_0(E)}\left\{\int_E |\nabla \varphi|^p\,dx\, :\, \varphi \ge 1 \text{ on } \Sigma\right\}.
\]
\end{definition}
\begin{remark}
\label{oss:altrefunzioni}
By using standard approximation methods, it is not difficult to see that the infimum above does not change, if we replace $C^\infty_0(E)$ by the space of Lipschitz functions, compactly supported in $E$. We indicate this space by $\mathrm{Lip}_0(E)$. We observe that, for every $\varphi\in \mathrm{Lip}_0(E)$ with $\varphi \ge 1$ on $\Sigma$, the new function
\[
\widetilde{\varphi}:=\min\{|\varphi|,1\},
\]
still belongs to $\mathrm{Lip}_0(E)$ and is such that 
\[
\int_E |\nabla \widetilde{\varphi}|^p\,dx\le \int_E |\nabla \varphi|^p\,dx,\qquad 0\le \widetilde\varphi\le 1\qquad \text{and}\qquad \widetilde\varphi =1 \text{ on } \Sigma.
\]
This shows that we also have the following equivalent characterization
\begin{equation}
\label{capalt}
\mathrm{cap}_p(\Sigma;E)=\inf_{\varphi\in \mathrm{Lip}_0(E)}\left\{\int_E |\nabla \varphi|^p\,dx\, :\, 0\le \varphi\le 1,\ \varphi =1 \text{ on } \Sigma\right\}.
\end{equation}
\end{remark}

\subsection{Capacity of balls}
For $N\ge 2$, we recall the expression for the $p-$capacity of a ball relative to a concentric ball (see \cite[page 148]{Maz}). Due to translation invariance, we can suppose that all the balls are centered at the origin.
We have to distinguish the cases $p=1$, $p\in(1,N)\cup(N,\infty)$ or $p=N$. For every $0<r<R$, this is given by\footnote{The reader should keep in mind that in \cite{Maz} the constant $\omega_N$ stands for the perimeter of $B_1(0)$, rather than for its volume. This explains the apparent difference with the formulas here given.}
\begin{equation}
\label{cappalla0}
\mathrm{cap}_1\left(\overline{B_r};B_{R}\right)=N\,\omega_N\,r^{N-1},
\end{equation}
\begin{equation}
\label{cappalla}
\mathrm{cap}_p\left(\overline{B_r};B_{R}\right)=N\,\omega_N\, \left|\frac{N-p}{p-1}\right|^{p-1}\,\frac{r^{N-p}}{\left|1-\left(\dfrac{r}{R}\right)^\frac{N-p}{p-1}\right|^{p-1}}
, \qquad \text{ if } p\in(1,N)\cup(N,\infty),
\end{equation}
and 
\begin{equation}
\label{cappallaconf}
\mathrm{cap}_N\left(\overline{B_r};B_{R}\right)=N\, \omega_N\, \left(\log \left(\frac{R}{r}\right)\right)^{1-N}.
\end{equation}
For $p>N$, we can even take the limit as $r$ goes to $0$ and get
\begin{equation}
\label{cappunto}
\mathrm{cap}_p\left(\{0\};B_{R}\right)=N\,\omega_N\, \left(\frac{p-N}{p-1}\right)^{p-1}\,R^{N-p}, \qquad \text{ if } p>N.
\end{equation}
\begin{remark}
\label{oss:scaling}
We observe in particular that we have the following scaling relations
\[
\mathrm{cap}_p\left(\overline{B_r};B_{R}\right)=r^{N-p}\,\mathrm{cap}_p\left(\overline{B_1};B_{R/r}\right), \qquad \text{ if } 1\le p<\infty,
\]
\end{remark}
The following technical result will be useful. It is a sort of ``quantified'' monotonicity inequality for the relative $p-$capacity of balls, with a geometric remainder term.
\begin{lemma}
\label{lm:cazzatella}
Let $N\ge 2$ and $1<p\le N$, for every $0<r_1<r_2< R$ we have
\[
\frac{|B_{r_2}\setminus B_{r_1}|}{(\mathcal{H}^{N-1}(\partial B_{r_2}))^\frac{p}{p-1}} 
+\left(\frac{1}{\mathrm{cap}_p(\overline{B_{r_2}};B_R)}\right)^\frac{1}{p-1}\le \left(\frac{1}{\mathrm{cap}_p(\overline{B_{r_1}};B_R)}\right)^\frac{1}{p-1}.
\]
\end{lemma}
\begin{proof}
The proof is just based on writing explicitly all the involved quantities and then using a convexity inequality. We start from the case $1<p<N$: by using \eqref{cappalla}, the claimed inequality is equivalent to
\[
\begin{split}
\frac{1}{(N\,\omega_N)^\frac{1}{p-1}}&\, \left(\frac{p-1}{N-p}\right)\,\frac{1-\left(\dfrac{r_1}{R}\right)^\frac{N-p}{p-1}}{\left(\dfrac{r_1}{R}\right)^\frac{N-p}{p-1}}\,\frac{1}{R^\frac{N-p}{p-1}}\\
&\ge \frac{1}{(N\,\omega_N)^\frac{1}{p-1}}\, \left(\frac{p-1}{N-p}\right)\,\frac{1-\left(\dfrac{r_2}{R}\right)^\frac{N-p}{p-1}}{\left(\dfrac{r_2}{R}\right)^\frac{N-p}{p-1}}\,\frac{1}{R^\frac{N-p}{p-1}}+\frac{1}{N\,(N\,\omega_N)^\frac{1}{p-1}}\,\frac{1}{r_2^\frac{N-p}{p-1}}\,\left(1-\left(\frac{r_1}{r_2}\right)^N\right).
\end{split}
\]
In turn, this is equivalent to the following inequality
\[
\begin{split}
\frac{1-\left(\dfrac{r_1}{R}\right)^\frac{N-p}{p-1}}{\left(\dfrac{r_1}{R}\right)^\frac{N-p}{p-1}}&\ge \frac{1-\left(\dfrac{r_2}{R}\right)^\frac{N-p}{p-1}}{\left(\dfrac{r_2}{R}\right)^\frac{N-p}{p-1}}+\frac{N-p}{N\,(p-1)}\,\left(\frac{R}{r_2}\right)^\frac{N-p}{p-1}\,\left(1-\left(\frac{r_1}{r_2}\right)^N\right),
\end{split}
\]
which is the same as
\[
\left(\dfrac{R}{r_1}\right)^\frac{N-p}{p-1}\ge \left(\dfrac{R}{r_2}\right)^\frac{N-p}{p-1}+\frac{N-p}{N\,(p-1)}\,\left(\frac{R}{r_2}\right)^\frac{N-p}{p-1}\,\left(1-\left(\frac{r_1}{r_2}\right)^N\right).
\]
This can be further rewritten as follows
\begin{equation}
\label{perseverance}
\left(\dfrac{R}{r_1}\right)^\frac{N-p}{p-1}\ge \left(\frac{N-p}{N\,(p-1)}+1\right)\,\left(\dfrac{R}{r_2}\right)^\frac{N-p}{p-1}-\frac{N-p}{N\,(p-1)}\,\left(\frac{R}{r_2}\right)^\frac{N-p}{p-1}\,\left(\frac{r_1}{r_2}\right)^N
\end{equation}
We now introduce the following notation
\[
t=\left(\dfrac{R}{r_1}\right)^N,\qquad s=\left(\dfrac{R}{r_2}\right)^N,\qquad \alpha=\frac{N-p}{N\,(p-1)}+1.
\]
In light of this notation, the above inequality \eqref{perseverance}  can be written as
\[
t^{\alpha-1}\ge \alpha\,s^{\alpha-1}-(\alpha-1)\,s^{\alpha-1}\,\frac{s}{t}.
\]
By multiplying both sides by the positive number $t$, the latter is equivalent to
\[
t^\alpha\ge \alpha\,s^{\alpha-1}\,t-(\alpha-1)\,s^\alpha,
\]
that is
\[
t^\alpha\ge s^\alpha+\alpha\,s^{\alpha-1}\,(t-s).
\]
We finally observe that this last inequality holds true for every $t,s\ge 0$, since this is nothing but the ``above tangent'' property of the convex function $\tau\mapsto \tau^\alpha$ (recall that $\alpha>1$, by definition). This concludes the proof for the case $1<p<N$.
\vskip.2cm\noindent
For the case $p=N$, one could simply observe that for every $0<r<R$, we have 
\[
\begin{split}
\lim_{p\nearrow N}\left(\frac{1}{\mathrm{cap}_p(\overline{B_{r}};B_R)}\right)^\frac{1}{p-1}&=\lim_{p\nearrow N}\frac{1}{(N\,\omega_N)^\frac{1}{p-1}}\, \left(\frac{p-1}{N-p}\right)\,\frac{1-\left(\dfrac{r}{R}\right)^\frac{N-p}{p-1}}{\left(\dfrac{r}{R}\right)^\frac{N-p}{p-1}}\,\frac{1}{R^\frac{N-p}{p-1}}\\
&=\frac{1}{(N\,\omega_N)^\frac{1}{N-1}}\,\lim_{p\nearrow N}\left(\frac{p-1}{N-p}\right)\,\left[\left(\dfrac{R}{r}\right)^\frac{N-p}{p-1}-1\right]\\
&=\frac{1}{(N\,\omega_N)^\frac{1}{N-1}}\,\left(\log \left(\frac{R}{r}\right)\right)=\left(\frac{1}{\mathrm{cap}_N(\overline{B_{r}};B_R)}\right)^\frac{1}{N-1}.
\end{split}
\]
Thus, it is sufficient to take the limit in the inequality for the case $1<p<N$, in order to conclude.
\end{proof}
\begin{remark}
\label{rem:cazzatella}
If $1<p<\infty$ and $\Sigma$ is a compact subset of the open set $E\subseteq \mathbb{R}^N$, the quantity
\[
\left(\frac{1}{\mathrm{cap}_p(\Sigma;E)}\right)^\frac{1}{p-1},
\]
is sometimes called the {\it $p-$modulus of $\Sigma$ relative to $E$}, see for example \cite[Chapter 2]{Flu}. In light of this, the estimate of Lemma \ref{lm:cazzatella} could also be seen as a consequence of the subadditivity of the $p-$modulus, which in the case of concentric balls reads as follows
\[
\left(\frac{1}{\mathrm{cap}_p(\overline{B_{r_1}};B_{r_2})}\right)^\frac{1}{p-1}
+\left(\frac{1}{\mathrm{cap}_p(\overline{B_{r_2}};B_R)}\right)^\frac{1}{p-1}\le \left(\frac{1}{\mathrm{cap}_p(\overline{B_{r_1}};B_R)}\right)^\frac{1}{p-1},
\] 
see \cite[Lemma 2.1]{Flu}. Indeed, by using the explicit expression of the quantities involved, it is not too difficult to see that 
\[
\left(\frac{1}{\mathrm{cap}_p(\overline{B_{r_1}};B_{r_2})}\right)^\frac{1}{p-1}\ge \frac{|B_{r_2}\setminus B_{r_1}|}{\big(\mathcal{H}^{N-1}(\partial B_{r_2})\big)^{\frac{p}{p-1}}}.
\]
We preferred to give here an elementary proof of the estimate which is needed for our purposes.
\par
For $p=2$, the previous subadditivity property of the $p-$modulus is also known as {\it Gr\"otzsch's lemma} (see for example \cite[Lemma 1.2]{Dub} and \cite[page 52, equation (8)]{PS}). 
\end{remark}

\section{Analysis of a Poincar\'e--type constant in a ball}
\label{sec:3}

The following result will be expedient in order to get the upper bound of the Main Theorem. The main point is the identity \eqref{secret} below.
\begin{lemma}
\label{lm:V}
Let $N\ge 2$ and $1<p\le N$. For $0<r_1<r_2< R$, we set
\[
S_{r_1,r_2}:=B_{r_2}\setminus \overline{B_{r_1}}=\Big\{x\in\mathbb{R}^N\, :\, r_1<|x|<r_2\Big\}.
\]
Let $V$ be the unique minimizer of the following problem
\[
\min_{\varphi\in W^{1,p}_0(B_R)} \left\{\frac{1}{p}\,\int_{B_R} |\nabla \varphi|^p\,dx-\int_{S_{r_1,r_2}} \varphi\,dx\right\}.
\]
Then $V$ is a $C^1(\overline{B_R})$ radially symmetric non-increasing function and it satisfies
\begin{equation}
\label{EL-V}
\int_{B_R} \langle |\nabla V|^{p-2}\,\nabla V,\nabla \varphi\rangle\,dx=\int_{S_{r_1,r_2}} \varphi\,dx,\qquad \text{ for every } \varphi\in W^{1,p}_0(B_R).
\end{equation}
Moreover, we have 
\begin{equation}
\label{secret}
\begin{split}
\int_{B_R} |\nabla V|^p\,dx&=\int_{S_{r_1,r_2}} \left(\frac{|x|}{N}\,\left(1-\left(\frac{r_1}{|x|}\right)^N\right)\right)^\frac{p}{p-1}dx+|S_{r_1,r_2}|^\frac{p}{p-1}\,\left(\frac{1}{\mathrm{cap}_p(\overline{B_{r_2}};B_R)}\right)^\frac{1}{p-1}.
\end{split}
\end{equation}
\end{lemma}
\begin{proof}
Existence of a minimizer follows from the Direct Method in the Calculus of Variations. Uniqueness is a consequence of the strict convexity of the functional which is minimized.
Finally, we can observe that \eqref{EL-V} is the Euler-Lagrange equation of this minimization problem, thus $V$ satisfies it just by minimality. We can also infer that $V\in C^{1,\alpha}(\overline{B_R})$ for some $0<\alpha<1$, thanks to the classical regularity result \cite[Theorem 1]{Lie88}.
\par
The radial symmetry of $V$ follows from its uniqueness and the fact that the data of the problem are rotationally invariant. Thus, we must have 
\[
V(x)=v(|x|),\qquad \text{ for } x\in B_R,
\]
where $v$ is a function of one variable. We want to prove that $v$ is non-increasing: at this aim, we set
\[
\widetilde{v}(t)=\int_t^R |v'(\tau)|\,d\tau,\qquad \text{ for } t\in (0,R).
\]
Thus, by definition $\widetilde{v}$ is non-increasing. Moreover, we have 
\[
\widetilde{v}'(t)=-|v'(t)|,\qquad \text{for}\ t\in(0,R),
\]
and 
\[
\widetilde{v}(t)=\int_t^R |v'(\tau)|\,d\tau\ge \left|\int_t^R v'(\tau)\,d\tau\right|=|v(t)|\ge v(t).
\]
These facts show that if we set $\widetilde{V}(x)=\widetilde{v}(|x|)$, then 
\[
\frac{1}{p}\,\int_{B_R} |\nabla \widetilde{V}|^p\,dx-\int_{S_{r_1,r_2}} \widetilde{V}\,dx\le \frac{1}{p}\,\int_{B_R} |\nabla V|^p\,dx-\int_{S_{r_1,r_2}} V\,dx. 
\]
By minimality of $V$, we must have $V=\widetilde{V}$ and thus the claimed monotonicity follows.
\vskip.2cm\noindent
We now need to prove formula \eqref{secret}. We observe at first that by testing \eqref{EL-V} with $\varphi=V$, we obtain
\begin{equation}
\label{testate}
\int_{B_R} |\nabla V|^p\,dx=\int_{S_{r_1,r_2}} V\,dx.
\end{equation}
Still from \eqref{EL-V}, we get in particular
\[
\int_{B_R} \langle |\nabla V|^{p-2}\,\nabla V,\nabla \varphi\rangle\,dx=0,\qquad \text{ for every } \varphi\in W^{1,p}_0(B_{r_1}).
\]
Thus, the function $V$ is weakly $p-$harmonic in the ball $B_{r_1}$. Moreover, thanks to its radial symmetry, it is constant on $\partial B_{r_1}$. By uniqueness of the Dirichlet problem for the $p-$Laplacian, we obtain that $V$ must be constant on the whole $B_{r_1}$. Thus, we obtain
\begin{equation}
\label{presplit}
\int_{B_R} |\nabla V|^p\,dx=\int_{B_R\setminus B_{r_1}} |\nabla V|^p\,dx.
\end{equation}
In turn, we split the last integral in two parts
\begin{equation}
\label{split}
\int_{B_R\setminus B_{r_1}} |\nabla V|^p\,dx=\int_{S_{r_1,r_2}} |\nabla V|^p\,dx+\int_{B_R\setminus B_{r_2}} |\nabla V|^p\,dx.
\end{equation}
In order to determine the first integral on the right-hand side, we take $h\in C^\infty_0((r_1,r_2))$ and use \eqref{EL-V} with test function $\varphi(x)=h(|x|)$. By using spherical coordinates and recalling the notation $V(x)=v(|x|)$, we get
\[
\int_{r_1}^{r_2} |v'|^{p-2}\,v'\,h'\,\varrho^{N-1}\,d\varrho=\int_{r_1}^{r_2} h\,\varrho^{N-1}\,d\varrho.
\]
We integrate by parts the last term, so to obtain
\[
\int_{r_1}^{r_2} \left[|v'|^{p-2}\,v'\,\varrho^{N-1}+\frac{\varrho^N}{N}\right]\,h'\,d\varrho=0,\qquad \text{ for every }h\in C^\infty_0((r_1,r_2)).
\]
This implies that there exists a constant $C$ such that
\[
|v'|^{p-2}\,v'\,\varrho^{N-1}+\frac{\varrho^N}{N}=C,\qquad \text{ on } (r_1,r_2).
\]
By recalling that $v'\le 0$, from this identity we get
\[
(-v'(\varrho))^{p-1}=\frac{\varrho}{N}-\frac{C}{\varrho^{N-1}},\qquad \text{ for } \varrho\in(r_1,r_2).
\]
The constant $C$ can be determined, by recalling that $v$ is $C^1$ and that $v$ is constant on $[0,r_1]$, from the above discussion. Thus, it must result
\[
0=(-v'(r_1))^{p-1}=\frac{r_1}{N}-\frac{C}{r_1^{N-1}}\qquad \text{ that is }\qquad C=\frac{r_1^N}{N}.
\]
In conclusion, we get that 
\[
|\nabla V(x)|^{p}=(-v'(|x|))^p=\left(\frac{|x|}{N}-\frac{r_1}{N}\,\left(\frac{r_1}{|x|}\right)^{N-1}\right)^\frac{p}{p-1}=\left(\frac{|x|}{N}\,\left(1-\left(\frac{r_1}{|x|}\right)^N\right)\right)^\frac{p}{p-1}.
\]
By integrating it over $S_{r_1,r_2}$, we get
\begin{equation}
\label{I1}
\int_{S_{r_1,r_2}} |\nabla V|^p\,dx=\int_{S_{r_1,r_2}}\left(\frac{|x|}{N}\,\left(1-\left(\frac{r_1}{|x|}\right)^N\right)\right)^\frac{p}{p-1}\,dx.
\end{equation}
We still need to determine the second integral in \eqref{split}.
To this aim, we take for every $n\in\mathbb{N}$ sufficiently large, the following function
\[
\varphi_n=1_{S_{r_1,r_2}}\ast \rho_n,
\]
where $\{\rho_n\}_{n\in\mathbb{N}\setminus\{0\}}$ is the usual family of radial smoothing kernels. By using \eqref{EL-V} with $\varphi=V\,\varphi_n$, we get
\[
\int_{B_R} |\nabla V|^{p}\,\varphi_n\,dx+\int_{B_R} \langle |\nabla V|^{p-2}\,\nabla V,\nabla \varphi_n\rangle\,V\,dx=\int_{S_{r_1,r_2}} V\,\varphi_n\,dx.
\]
By using the properties of convolutions, the regularity of $V$ and the radial symmetry of both $V$ and $\varphi_n$, the limit as $n$ goes to $\infty$ yields
\[
\int_{S_{r_1,r_2}} |\nabla V|^{p}\,dx+v(r_2)\,(-v'(r_2))^{p-1}\,\mathcal{H}^{N-1}(\partial B_{r_2})=\int_{S_{r_1,r_2}} V\,dx.
\]
Observe that we also used that $v'(r_1)=0$, as explained above. By using \eqref{testate}, \eqref{presplit} and \eqref{split}, this in turn implies
\[
\cancel{\int_{S_{r_1,r_2}} |\nabla V|^{p}\,dx}+v(r_2)\,(-v'(r_2))^{p-1}\,\mathcal{H}^{N-1}(\partial B_{r_2})=\cancel{\int_{S_{r_1,r_2}} |\nabla V|^p\,dx}+\int_{B_R\setminus B_{r_2}} |\nabla V|^p\,dx.
\]
That is
\[
\int_{B_R\setminus B_{r_2}} |\nabla V|^p\,dx=v(r_2)\,(-v'(r_2))^{p-1}\,\mathcal{H}^{N-1}(\partial B_{r_2}).
\]
The term $v'(r_2)$ can be computed, thanks to the fact that $v$ is $C^1$ and to the exact determination of $v'$ on the interval $(r_1,r_2)$. We must have
\[
(-v'(r_2))^{p-1}=\frac{r_2}{N}-\frac{r_1}{N}\,\left(\frac{r_1}{r_2}\right)^{N-1}=\frac{r_2}{N}\,\left(1-\left(\frac{r_1}{r_2}\right)^N\right).
\]
Thus, up to now we have obtained
\begin{equation}
\label{quasirotto}
\int_{B_R\setminus B_{r_2}} |\nabla V|^p\,dx=v(r_2)\,\frac{r_2}{N}\,\left(1-\left(\frac{r_1}{r_2}\right)^N\right)\,\mathcal{H}^{N-1}(\partial B_{r_2})=v(r_2)\,|S_{r_1,r_2}|.
\end{equation}
We still need to determine $v(r_2)$. To this aim, it is sufficient to observe that, thanks to both the monotonicity and the $p-$harmonicity of $V$, the function 
\[
W=\min\left\{\frac{V}{v(r_2)},\,1\right\},
\] 
is a weakly $p-$harmonic function in $B_R\setminus \overline{B_{r_2}}$, vanishing on $\partial B_R$ and is equal to $1$ on $\overline{B_{r_2}}$. Thus, it must be the $p-$capacitary potential of $\overline{B_{r_2}}$, relative to $B_R$, i.e. we have\footnote{It is not difficult to see that there exists a sequence $\{W_n\}_{n\in\mathbb{N}}\subseteq \mathrm{Lip}_0(B_R)$ such that $0\le W_n\le 1$, $W_n\equiv 1$ on $\overline{B_{r_2}}$ and
\[
\lim_{n\to\infty} \|\nabla W_n\|_{L^p(B_R)}=\|\nabla W\|_{L^p(B_R)}.
\]
Thus, in light of \eqref{capalt}, we have 
\[
\mathrm{cap}_p(\overline{B_{r_2}};B_R)\le \lim_{n\to\infty} \int_{B_R} |\nabla W_n|^p\,dx=\int_{B_R} |\nabla W|^p\,dx=\int_{B_R\setminus B_{r_2}} |\nabla W|^p\,dx.
\]
On the other hand, for every $\varphi\in\mathrm{Lip}_0(B_R)$ such that $0\le \varphi\le 1$ and $\varphi\equiv 1$ on $\overline{B_{r_2}}$, we have 
\[
\int_{B_R} |\nabla \varphi|^p\,dx=\int_{B_R\setminus B_{r_2}} |\nabla \varphi|^p\,dx\ge \int_{B_R\setminus B_{r_2}} |\nabla W|^p\,dx+p\,\int_{B_R\setminus B_{r_2}} \langle |\nabla W|^{p-2}\,\nabla W,\nabla \varphi-\nabla W\rangle\,dx.
\]
By using \eqref{EL-V}, the fact that $\varphi-W\in W^{1,p}_0(B_R)$ and $\varphi-W\equiv0$ on $\overline{B_{r_2}}$, we get that the rightmost integral vanishes. By arbitrariness of $\varphi$ and using again \eqref{capalt}, we obtain the claimed identity.}
\[
\int_{B_R\setminus B_{r_2}} |\nabla W|^p=\mathrm{cap}_p(\overline{B_{r_2}};B_R).
\]
This is the same as 
\[
\int_{B_R\setminus B_{r_2}} |\nabla V|^p\,dx=v(r_2)^p\,\mathrm{cap}_p(\overline{B_{r_2}};B_R).
\]
By comparing the previous  two expressions for $\int_{B_R\setminus B_{r_2}} |\nabla V|^p\,dx$, we get
\[
v(r_2)\,|S_{r_1,r_2}|=v(r_2)^p\,\mathrm{cap}_p(\overline{B_{r_2}};B_R).
\]
This finally gives
\[
v(r_2)=\left(\frac{|S_{r_1,r_2}|}{\mathrm{cap}_p(\overline{B_{r_2}};B_R)}\right)^\frac{1}{p-1}.
\]
By using this expression in \eqref{quasirotto}, we end up with 
\begin{equation}
\label{I2}
\int_{B_R\setminus B_{r_2}} |\nabla V|^p\,dx=|S_{r_1,r_2}|^\frac{p}{p-1}\,\left(\frac{1}{\mathrm{cap}_p(\overline{B_{r_2}};B_R)}\right)^\frac{1}{p-1}.
\end{equation}
By using \eqref{I1} and \eqref{I2} in \eqref{split} and recalling \eqref{presplit}, we finally obtain the desired formula.
\end{proof}
As a consequence of the properties of the function $V$, we can estimate a suitable Poincar\'e--type constant. This is the main result of this section, contained in the following
\begin{proposition}
\label{coro:sharp1p}
Let $N\ge 2$ and $1 < p\le N$. With the same notation of Proposition \ref{lm:V}, we have
\begin{equation}
\label{torsiontype1}
\sup_{\varphi\in W^{1,p}_0(B_R)\setminus\{0\}} \frac{\displaystyle\left(\int_{S_{r_1,r_2}} |\varphi|\,dx\right)^p}{\displaystyle \int_{B_R}|\nabla \varphi|^p\,dx}=\left(\int_{B_R} |\nabla V|^p\,dx\right)^{p-1}.
\end{equation}
In particular, for every $\varphi\in W^{1,p}_0(B_R)$ we get
\begin{equation}
\label{torsiontype2}
\begin{split}
\left(\fint_{S_{r_1,r_2}} |\varphi|\,dx\right)^p&\le \left[\frac{|S_{r_1,r_2}|}{(\mathcal{H}^{N-1}(\partial B_{r_2}))^\frac{p}{p-1}} 
+\left(\frac{1}{\mathrm{cap}_p(\overline{B_{r_2}};B_R)}\right)^\frac{1}{p-1}\right]^{p-1}\,\int_{B_R}|\nabla \varphi|^p\,dx.
\end{split}
\end{equation}
\end{proposition}
\begin{proof}
By using $V$ as a test function, we have 
\[
\sup_{\varphi\in W^{1,p}_0(B_R)\setminus\{0\}} \frac{\displaystyle\left(\int_{S_{r_1,r_2}} |\varphi|\,dx\right)^p}{\displaystyle \int_{B_R}|\nabla \varphi|^p\,dx}\ge \frac{\displaystyle\left(\int_{S_{r_1,r_2}} V\,dx\right)^p}{\displaystyle \int_{B_R}|\nabla V|^p\,dx}=\left(\int_{B_R} |\nabla V|^p\,dx\right)^{p-1}.
\]
We also used the identity \eqref{testate}. On the other hand, by taking $\varphi\in W^{1,p}_0(B_R)$ and testing the equation \eqref{EL-V} with $|\varphi|\in W^{1,p}_0(B_R)$, we get
\[
\int_{S_{r_1,r_2}} |\varphi|\,dx=\int_{B_R} \langle |\nabla V|^{p-2}\,\nabla V,\nabla |\varphi|\rangle\,dx\le \left(\int_{B_R} |\nabla V|^p\,dx\right)^\frac{p-1}{p}\,\left(\int_{B_R} |\nabla \varphi|^p\,dx\right)^\frac{1}{p}.
\]
The desired conclusion \eqref{torsiontype1} now follows, thanks to the arbitrariness of $\varphi\in W^{1,p}_0(B_R)$. 
\par
The estimate \eqref{torsiontype2} will simply follow from \eqref{torsiontype1}, once we recall the expression \eqref{secret} for the $L^p$ norm of $\nabla V$. We estimate from above the latter:
observe that the function
\[
t\mapsto \left(\frac{t}{N}\right)^\frac{p}{p-1}\,\left(1-\left(\frac{r_1}{t}\right)^N\right)^\frac{p}{p-1},
\]
is monotone increasing. Thus, the estimate \eqref{secret} implies
\[
\begin{split}
\int_{B_R} |\nabla V|^p\,dx&\le \left(\frac{r_2}{N}\,\left(1-\left(\frac{r_1}{r_2}\right)^N\right)\right)^\frac{p}{p-1}\, |S_{r_1,r_2}| +\left(\frac{|S_{r_1,r_2}|^{p}}{\mathrm{cap}_p(\overline{B_{r_2}};B_R)}\right)^\frac{1}{p-1}.\\
\end{split}
\]
We then observe that 
\[
\left(\frac{r_2}{N}\,\left(1-\left(\frac{r_1}{r_2}\right)^N\right)\right)^\frac{p}{p-1}=\left(\frac{|B_{r_2}|}{\mathcal{H}^{N-1}(\partial B_{r_2})}\,\frac{|S_{r_1,r_2}|}{|B_{r_2}|}\right)^\frac{p}{p-1}.
\]
Thus, we get
\[
\begin{split}
\frac{\left(\displaystyle\int_{B_R} |\nabla V|^p\,dx\right)^{p-1}}{|S_{r_1,r_2}|^p}
&\le \left[\frac{|S_{r_1,r_2}|}{(\mathcal{H}^{N-1}(\partial B_{r_2}))^\frac{p}{p-1}} 
+\left(\frac{1}{\mathrm{cap}_p(\overline{B_{r_2}};B_R)}\right)^\frac{1}{p-1}\right]^{p-1}.
\end{split}
\]
This concludes the proof.
\end{proof}	
\begin{remark}
We observe that, by using the geometric estimate of Lemma \ref{lm:cazzatella} in \eqref{torsiontype2}, one can also get the slightly rougher (but definitely handier) estimate 
\begin{equation}
\label{torsiontype3}
\left(\fint_{S_{r_1,r_2}} |\varphi|\,dx\right)^p\le \frac{1}{\mathrm{cap}_p(\overline{B_{r_1}};B_R)}\,\int_{B_R}|\nabla \varphi|^p\,dx,\qquad \text{for every}\ \varphi\in W^{1,p}_0(B_R).
\end{equation}
\end{remark}
By a limiting argument, we can cover the case $p=1$, as well. In this case, the sharp constant has a simpler and nicer expression.
\begin{corollary}
\label{coro:sharp11}
Let $N\ge 2$. With the same notation of Proposition \ref{lm:V}, we have
\begin{equation}
\label{torsiontype21}
\begin{split}
\sup_{\varphi\in W^{1,1}_0(B_R)\setminus\{0\}} \frac{\displaystyle\int_{S_{r_1,r_2}} |\varphi|\,dx}{\displaystyle \int_{B_R}|\nabla \varphi|\,dx}=\frac{|S_{r_1,r_2}|}{\mathrm{cap}_1(\overline{B_{r_2}};B_R)}.
\end{split}
\end{equation}
\end{corollary}
\begin{proof}
As above, we take for every $n\in\mathbb{N}$ the following function
\[
\varphi_n=1_{B_{r_2}}\ast \rho_n,
\]
where $\{\rho_n\}_{n\in\mathbb{N}\setminus\{0\}}$ is the usual family of radial smoothing kernels.
Since $B_{r_2} \Subset B_R$, for $n$ sufficiently large we have that $\varphi_n\in C^\infty_0(B_R)$. By the properties of convolutions and by \cite[page 121]{AFP}, we have
\[
\sup_{\varphi\in W^{1,1}_0(B_R)\setminus\{0\}} \frac{\displaystyle\int_{S_{r_1,r_2}} |\varphi|\,dx}{\displaystyle \int_{B_R}|\nabla \varphi|\,dx}\ge\lim_{n \to\infty} \frac{\displaystyle \int_{S_{r_1,r_2}} |\varphi_n| \,dx} {\displaystyle \int_{B_R} |\nabla \varphi_n|\,dx}=\dfrac{|S_{r_1,r_2}|}{\mathcal{H}^{N-1}(\partial B_{r_2})}=\dfrac{|S_{r_1,r_2}|}{\mathrm{cap}_1(\overline{B_{r_2}};B_R)}.
\]
In order to prove the reverse inequality, we first observe that 
\[
\lim_{p\searrow 1} \frac{1}{\mathrm{cap}_p(\overline{B_{r_2}};B_R)}=\frac{1}{\mathrm{cap}_1(\overline{B_{r_2}};B_R)}.
\]
This simply follows by recalling the expressions \eqref{cappalla} and \eqref{cappalla0}. We also claim that 
\[
\lim_{p\searrow 1}\left[\frac{|S_{r_1,r_2}|}{\mathcal{H}^{N-1}(\partial B_{r_2})}\,\left(\frac{\mathrm{cap}_p(\overline{B_{r_2}};B_R)}{\mathcal{H}^{N-1}(\partial B_{r_2})}\right)^\frac{1}{p-1}+1\right]^{p-1}=1.
\]
Indeed, we have
\[
\frac{|S_{r_1,r_2}|}{\mathcal{H}^{N-1}(\partial B_{r_2})}\,\left(\frac{\mathrm{cap}_p(\overline{B_{r_2}};B_R)}{\mathcal{H}^{N-1}(\partial B_{r_2})}\right)^\frac{1}{p-1}=\frac{r_2^N-r_1^N}{N\,r_2^N}\,\left(\frac{N-p}{p-1}\right)\,\frac{1}{\left(1-\left(\dfrac{r_2}{R}\right)^\frac{N-p}{p-1}\right)}.
\]
Thus, 
for $p$ converging to $1$
\[
\begin{split}
(p-1)\,&\log \left(\frac{|S_{r_1,r_2}|}{\mathcal{H}^{N-1}(\partial B_{r_2})}\,\left(\frac{\mathrm{cap}_p(\overline{B_{r_2}};B_R)}{\mathcal{H}^{N-1}(\partial B_{r_2})}\right)^\frac{1}{p-1}+1\right)\\
&=(p-1)\,\log\left(\frac{r_2^N-r_1^N}{N\,r_2^N}\,\left(\frac{N-p}{p-1}\right)\,\frac{1}{\left(1-\left(\dfrac{r_2}{R}\right)^\frac{N-p}{p-1}\right)}+1\right)\sim (p-1)\,\log\left(\frac{N-p}{p-1}\,\right).
\end{split}
\]
Since the last quantity is infinitesimal, we get the claimed limit.
\par
We now take $\varphi\in C^\infty_0(B_R)\setminus\{0\}$. By taking the limit as $p$ goes to $1$ in \eqref{torsiontype2} and using the previous results, we get
\[
\frac{\displaystyle\int_{S_{r_1,r_2}} |\varphi|\,dx}{\displaystyle \int_{B_R}|\nabla \varphi|\,dx}\le \frac{|S_{r_1,r_2}|}{\mathrm{cap}_1(\overline{B_{r_2}};B_R)}.
\]
By arbitrariness of $\varphi$ and by density of $C^\infty_0(B_R)$ in $W^{1,1}_0(B_R)$, we get the conclusion.
\end{proof}
\begin{remark}[A Cheeger--type constant]
\label{oss:cheegertype}
It is not difficult to see that 
\[
\inf_{\varphi\in C^\infty_0(B_R)\setminus\{0\}} \frac{\displaystyle \int_{B_R}|\nabla \varphi|\,dx}{\displaystyle\int_{S_{r_1,r_2}} |\varphi|\,dx}=\inf\left\{\dfrac{\mathcal{H}^{N-1}( \partial E)}{|E \cap S_{r_1,r_2}|}\, : \, E \Subset B_R \text{ has smooth boundary}\right\},
\]
see for example \cite[Theorem 2.1.3]{Maz}. We tacitly assume that the last ratio is $+\infty$, for those sets $E$ such that $|E\cap S_{r_1,r_2}|=0$.
In light of Corollary \ref{coro:sharp11}, we thus have 
\[
\inf\left\{\dfrac{\mathcal{H}^{N-1}( \partial E)}{|E \cap S_{r_1,r_2}|}\, : \, E \Subset B_R \text{ has smooth boundary}\right\}=\frac{\mathrm{cap}_1(\overline{B_{r_2}};B_R)}{|S_{r_1,r_2}|}=\frac{\mathcal{H}^{N-1}(\partial B_{r_2})}{|S_{r_1,r_2}|}.
\]
In particular, we have that $E=B_{r_2}$ is an optimal shape, for this Cheeger--type constant. 
\end{remark}

\section{Proof of the Main Theorem: lower bound}
\label{sec:4}

We start with the following easy fact: this is a particular case of \cite[Chapter 13, Proposition 1, page 658]{Maz}. We report the proof for the reader's convenience. 
\begin{proposition} 
	Let $1 \leq p < \infty$ and let $\Sigma \subseteq B_r(x_0)$ be a compact set. Then, for every $R>r$ we have 
	\begin{equation} 
	\label{comparison_cap}
		\mathrm{cap}_p(\Sigma; B_R(x_0)) \leq \mathrm{cap}_p(\Sigma; B_r(x_0)) \leq \left(\frac{1}{\lambda_p(B_1)^{\frac{1}{p}}}\, \frac{R}{d}+1 \right)^{p}\, \mathrm{cap}_p(\Sigma; B_R(x_0)), 
	\end{equation}
	where $d:=\mathrm{dist}(\Sigma,\partial B_r(x_0))>0$.
	\end{proposition}
\begin{proof}
	The leftmost inequality is straightforward, we thus focus on proving the rightmost one. Without loss of generality, we can assume that $x_0 = 0$. 
Let $u \in C^\infty_0(B_R)$ be such that $u\ge 1$ on $\Sigma$. For every $0<\varepsilon<d/2$ we take the Lipschitz cut-off function, compactly supported in $B_r$, given by
\[
\eta(x)=\min\left\{\left(\frac{(r-\varepsilon)-|x|}{(r-\varepsilon)-(r-d)}\right)_+,\,1\right\}.
\]
Observe that by construction the function $\psi=\eta\,u$ is a Lipschitz function, compactly supported in $B_r$ and such that $\psi\ge 1$ on $\Sigma$.
Thus, this is an admissible function to test the definition of relative $p-$capacity, thanks to Remark \ref{oss:altrefunzioni}. By using Minkowski's inequality and the properties of $\eta$, we get
\begin{equation}
 \begin{split}
 \label{tony}
		\Big(\mathrm{cap}_p(\Sigma; B_{r})\Big)^\frac{1}{p} &\leq \frac{1}{d-\varepsilon}\, \|u\|_{L^p(B_{r})} + \|\nabla u\|_{L^p(B_{r})} \\
		&\le \frac{1}{d-\varepsilon}\, \|u\|_{L^p(B_R)} + \|\nabla u\|_{L^p(B_R)}\\
	&\leq \left(\frac{1}{d-\varepsilon}\, \frac{R}{\lambda_p(B_1)^{\frac{1}{p}}}+1 \right) \|\nabla u \|_{L^p(B_R)}.
	\end{split}
\end{equation}
In the third inequality we also used Poincar\'e's inequality for the set $B_R$.
By taking the limit as $\varepsilon$ goes to $0$ and using the arbitrariness of $u$, we get the claimed estimate.
\end{proof}
We use the previous result in combination with a Maz'ya-Poincar\'e type inequality and a tiling of $\mathbb{R}^N$ made by cubes. This leads to the lower bound of the Main Theorem.
\begin{theorem} 
\label{teo:lower_bound}
	Let $1 \leq p \leq N$, there exists a constant $\sigma_{N,p} >0$ such that for every $\Omega \subseteq \mathbb{R}^N$ open set and every $0<\gamma<1$, we have
	\begin{equation}
		\lambda_p(\Omega) \geq  \sigma_{N,p}\,\gamma\,\left(\frac{1}{R_{p,\gamma}(\Omega)}\right)^p.
	\end{equation}
\end{theorem}
\begin{proof}
We first observe that if $R_{p,\gamma}(\Omega) = +\infty$, then there is nothing to prove. Thus, let us
assume that $R_{p,\gamma}(\Omega) < +\infty$. Let $r > R_{p,\gamma}(\Omega)$ and let $u \in C^{\infty}_0(\Omega)$, extended by $0$ to the complement $\mathbb{R}^N\setminus\Omega$. For every $x_0 \in \mathbb{R}^N$, by definition of capacitary inradius  we have 
	 \begin{equation}
	 \label{diomadonna0}
	 \begin{split}
		\mathrm{cap}_p(\overline{B_r(x_0)} \setminus \Omega; B_{2r}(x_0)) &>\gamma\, \mathrm{cap}_p(\overline{B_r(x_0)}; B_{2r}(x_0))= \gamma\, \mathrm{cap}_p(\overline{B_1}; B_2)\, r^{N-p}.
	\end{split}
	\end{equation}
	The last identity simply follows from the scaling properties of the relative $p-$capacity (see Remark \ref{oss:scaling}).
We now consider the cube $Q_r(x_0)$ concentric with $B_r(x_0)$. We observe that $u$ is a $C^\infty$ function on $\overline{Q_r(x_0)}$, which vanishes on the compact subset $\overline{B_r(x_0)} \setminus \Omega\subseteq \overline{Q_r(x_0)}$. Thus, we can use the Maz'ya-Poincar\'e inequality of \cite[Theorem 14.1.2]{Maz} (more precisely, we use its slight variant of \cite[Theorem 2.5]{BozBra}) to infer that 
	\begin{equation*}
	\frac{\mathscr{C}}{r^{\frac{N}{p}}}\, \mathrm{cap}_p\left(\overline{B_r(x_0)} \setminus \Omega; B_{2\sqrt{N}r}(x_0)\right)^{\frac{1}{p}}\, \|u\|_{L^p(Q_{r}(x_0))} \leq \|\nabla u\|_{L^p(Q_{r}(x_0))},
	\end{equation*}
	where $\mathscr{C}$ is the same constant as in \cite[Theorem 2.5]{BozBra}. Furthermore, by applying \eqref{comparison_cap} we get 
	\[
	\frac{\mathscr{C}}{\left(\dfrac{2\,\sqrt{N}}{\lambda_p(B_1)^{\frac{1}{p}}}+1 \right)^{\frac{1}{p}}}\, \frac{1}{r^{\frac{N}{p}}}\, \mathrm{cap}_p\left(\overline{B_r(x_0)} \setminus \Omega; B_{2r}(x_0)\right)^{\frac{1}{p}}\, \|u\|_{L^p(Q_{r}(x_0))} \leq \|\nabla u\|_{L^p(Q_{r}(x_0))}.
	\]
We can further apply \eqref{diomadonna0}, in order to estimate from below the left-hand side. By raising to the power $p$, this gives
	\begin{equation} 
	\label{diomadonna}
		\frac{\mathscr{C}^p\,\mathrm{cap}_p(\overline{B_1}; B_2)}{\dfrac{2\,\sqrt{N}}{\lambda_p(B_1)^{\frac{1}{p}}}+1}\, \frac{\gamma}{r^p}\,\|u\|_{L^p(Q_{r}(x_0))} \leq \|\nabla u\|_{L^p(Q_{r}(x_0))}.
	\end{equation}
By using this estimate for a family of disjoint cubes having radius $r$ and tiling the whole space, summing up we get	
\[	
\frac{\mathscr{C}^p\,\mathrm{cap}_p(\overline{B_1}; B_2)}{\dfrac{2\,\sqrt{N}}{\lambda_p(B_1)^{\frac{1}{p}}}+1}\, \frac{\gamma}{r^p}\,\|u\|^p_{L^p(\Omega)} \leq \|\nabla u \|^p_{L^p(\Omega)}.
\]
This concludes the proof by arbitrariness of $u$.
	\end{proof} 
\begin{remark}	
By inspecting the proof above, we see that we get the following constant
\[
\sigma_{N,p}=\frac{\mathscr{C}^p\,\mathrm{cap}_p(\overline{B_1}; B_2)}{\left(\dfrac{2\,\sqrt{N}}{\lambda_p(B_1)^{\frac{1}{p}}}+1 \right)}.
\]
A possible value for the constant $\mathscr{C}=\mathscr{C}(N,p)>0$ can be found in \cite[Remark 2.6]{BozBra}, by taking $q=p$ there and $D/d=2\,\sqrt{N}$.
\end{remark}

\section{Proof of the Main Theorem: upper bound}
\label{sec:5}

\begin{theorem} 
\label{teo:upper_bound}
Let $1 \leq p \leq N$, $0<\gamma<1$ and let $\Omega\subseteq\mathbb{R}^N$ be an open set such that $R_{p,\gamma}(\Omega)<+\infty$. Then we have 
\[
\lambda_p(\Omega)\le C_{N,p,\gamma}\,\left(\frac{1}{R_{p,\gamma}(\Omega)}\right)^p.
\]
On the other hand, if $R_{p,\gamma}(\Omega)=+\infty$, then we have $\lambda_p(\Omega)=0$.
\end{theorem}
\begin{proof}
Let $\gamma\in(0,1)$ be fixed, we take a ball $B_r(x_0)$ such that 
\begin{equation}
\label{ipotesiF}
\mathrm{cap}_p\left(\overline{B_r(x_0)}\setminus\Omega;B_{2\,r}(x_0)\right)\le \gamma\,\mathrm{cap}_p\left(\overline{B_r(x_0)};B_{2\,r}(x_0)\right).
\end{equation}
We will show that for every such $r$, we can bound
\begin{equation}
\label{volio}
\lambda_p(\Omega)\le \frac{C_{N,p,\gamma}}{r^p}.
\end{equation}
By taking the supremum over the admissible $r$, we will eventually get the result. In particular, if $R_{p,\gamma}(\Omega)=+\infty$ the previous upper bound will prove that $\lambda_p(\Omega)=0$.
\vskip.2cm\noindent
We set for simplicity $F=\overline{B_r(x_0)}\setminus\Omega$. For every $\delta>0$, we take $\varphi_\delta\in \mathrm{Lip}_0(B_{2\,r}(x_0))$ such that 
\begin{equation}
\label{quasiottima}
0\le \varphi_\delta\le 1,\qquad \varphi_\delta = 1\ \text{on}\ F\qquad \text{ and }\qquad \int_{B_{2\,r}(x_0)} |\nabla \varphi_\delta|^p\,dx\le \delta+\mathrm{cap}_p\left(F;B_{2\,r}(x_0)\right).
\end{equation}
Such a function exists, by recalling \eqref{capalt}. Without loss of generality, we can suppose that $x_0$ coincides with the origin.
\par
We observe at first that by density we have
\[
\lambda_p(\Omega)=\inf_{\varphi\in C^\infty_0(\Omega)} \left\{\int_\Omega |\nabla \varphi|^p\,dx\, :\, \|\varphi\|_{L^p(\Omega)}=1\right\}=\inf_{\varphi\in W^{1,p}_0(\Omega)} \left\{\int_\Omega |\nabla \varphi|^p\,dx\, :\, \|\varphi\|_{L^p(\Omega)}=1\right\}.
\]
We fix $0<\varepsilon<1/2$ and take a Lipschitz cut-off function $\eta$ defined on $\overline{B_r}$, such that 
\[
0\le \eta\le 1,\quad \eta\equiv 1 \text{ on } B_{(1-\varepsilon)\,r},\quad \eta\equiv 0\ \text{on}\ \partial B_r,\quad \|\nabla \eta\|_{L^\infty}= \frac{1}{\varepsilon\,r}.
\]
We use the test function $\psi=(1-\varphi_\delta)\,\eta/\|(1-\varphi_\delta)\,\eta\|_{L^p(\Omega)}$ in the definition of $\lambda_p(\Omega)$. Observe that this is a feasible test function: indeed, by construction we have that $\psi$ is a Lipschitz function on the whole $\mathbb{R}^N$. Moreover, we have that 
\[
\psi\equiv 0\qquad \text{ on } \partial(B_r\cap \Omega)\subseteq (\partial B_r\cap \overline\Omega)\cup(\partial \Omega\cap \overline{B_r}). 
\]
More precisely, we have 
\[
(1-\varphi_\delta)\equiv 0\qquad \text{ on } F= \overline{B_r}\setminus\Omega\supseteq \partial \Omega\cap \overline{B_r},
\]
and
\[
\eta\equiv 0\qquad \text{ on } \partial B_r\supseteq \partial B_r\cap \overline\Omega.
\]
Thus, by \cite[Theorem 9.17 \& Remark 19]{Bre} we get that 
\[
\psi=(1-\varphi_\delta)\,\eta\in W^{1,p}_0(B_r\cap \Omega)\subseteq W^{1,p}_0(\Omega).
\]
This gives 
\[
\begin{split}
\lambda_p(\Omega)&\le \frac{\displaystyle \int_\Omega |(1-\varphi_\delta)\,\nabla \eta-\eta\,\nabla \varphi_\delta|^p\,dx}{\displaystyle \int_\Omega (1-\varphi_\delta)^p\,\eta^p\,dx}\le 2^{p-1}\,\frac{\displaystyle \int_\Omega (1-\varphi_\delta)^p\,|\nabla \eta|^p\,dx+\int_\Omega \eta^p\,|\nabla \varphi_\delta|^p\,dx}{\displaystyle \int_\Omega (1-\varphi_\delta)^p\,\eta^p\,dx}.
\end{split}
\]
By using the properties of $\eta$, we get
\[
\lambda_p(\Omega)\,\int_{B_{(1-\varepsilon)\,r}} (1-\varphi_\delta)^p\,dx\le 2^{p-1}\,\left[\frac{1}{\varepsilon^p\,r^p}\,\int_{B_r\setminus B_{(1-\varepsilon)\,r}} (1-\varphi_\delta)^p\,dx+\int_{B_r} |\nabla \varphi_\delta|^p\,dx\right].
\]
We also use \eqref{quasiottima} and \eqref{ipotesiF} on the right-hand side: this leads to 
\[
\lambda_p(\Omega)\,\int_{B_{(1-\varepsilon)\,r}} (1-\varphi_\delta)^p\,dx\le 2^{p-1}\,\left[\frac{\omega_N\,r^N}{\varepsilon^p\,r^p}\,(1-(1-\varepsilon)^N)+\delta +\gamma\,\mathrm{cap}_p\Big(\overline{B_r};B_{2\,r}\Big)\right].
\]
We also observe that by convexity of the map $\tau\mapsto \tau^N$ we have 
\[
t^N\ge 1+N\,(t-1),\qquad \text{ for } t\ge 0,
\]
and thus 
\[
(1-\varepsilon)^N\ge 1-N\,\varepsilon\qquad \text{ that is }\qquad 1-(1-\varepsilon)^N\le N\,\varepsilon.
\]
This leads us to 
\begin{equation}
\label{passo0}
\lambda_p(\Omega)\,\int_{B_{(1-\varepsilon)\,r}} (1-\varphi_\delta)^p\,dx\le \frac{2^{p-1}}{r^p}\,r^{N}\,\left[\frac{N\,\omega_N}{\varepsilon^{p-1}}+\delta +\gamma\,\mathrm{cap}_p\Big(\overline{B_1};B_{2}\Big)\right].
\end{equation}
Observe that we also used the scaling properties of the $p-$capacity.
We now wish to give a lower bound for the leftmost integral. To this aim, we set
\[
r_1=(1-\ell)\,r,\qquad r_2=\left(1-\varepsilon\right)\,r,
\]
with $1>\ell>\varepsilon>0$ and $0<\varepsilon<1/2$.
By defining the spherical shell
\[
S_{r_1,r_2}:=B_{r_2}\setminus \overline{B_{r_1}}=\left\{x\in\mathbb{R}^N\, :\, r_1<|x|<r_2\right\},
\]
from the estimate above we obviously get
\begin{equation}
\label{passo1}
\lambda_p(\Omega)\,\fint_{S_{r_1,r_2}} (1-\varphi_\delta)^p\,dx\le \frac{2^{p-1}}{r^p}\,\frac{r^N}{|S_{r_1,r_2}|}\,\left[\frac{N\,\omega_N}{\varepsilon^{p-1}}+\delta +\gamma\,\mathrm{cap}_p\Big(\overline{B_1};B_{2}\Big)\right].
\end{equation}
By Jensen's inequality we have\footnote{For $p=1$, this is not needed, of course.} 
\[
\fint_{S_{r_1,r_2}} (1-\varphi_\delta)^p\,dx\ge \left(\fint_{S_{r_1,r_2}}(1-\varphi_\delta)\,dx \right)^p.
\] 
By inserting this estimate in \eqref{passo1}, we get
\begin{equation}
\label{dai!}
\lambda_p(\Omega)\,\left(1-\fint_{S_{r_1,r_2}}\varphi_\delta\,dx \right)^p\le \frac{2^{p-1}}{r^p}\,\frac{r^N}{|S_{r_1,r_2}|}\,\left[\frac{N\,\omega_N}{\varepsilon^{p-1}}+\delta+\gamma\,\mathrm{cap}_p\Big(\overline{B_1};B_{2}\Big)\right].
\end{equation}
In order to conclude, it would be sufficient to show that 
\[
1-\fint_{S_{r_1,r_2}}\varphi_\delta\,dx\ge \frac{1}{C},
\]
for some positive constant $C$, depending only on $N,p$ and $\gamma$. This is the key point, where the results of Section \ref{sec:3} will be crucial.
We now distinguish the case $1<p\le N$ and the case $p=1$.
\vskip.2cm\noindent
{\bf 1. Case $1 < p \leq N$.} 
To this aim, we can use the Poincar\'e--type inequality of \eqref{torsiontype3}, with $R=2\,r$. This yields
\[
\begin{split}
\left(\fint_{S_{r_1,r_2}} \varphi_\delta\,dx\right)^p&\le \frac{1}{\mathrm{cap}_p(\overline{B_{r_1}};B_{2\,r})}\,\int_{B_{2\,r}}|\nabla \varphi_\delta|^p\,dx.
\end{split}
\]
By using again \eqref{quasiottima} and \eqref{ipotesiF} in order to estimate the rightmost integral, we then obtain
\[
\fint_{S_{r_1,r_2}} \varphi_\delta\,dx\le\left(\frac{\delta+\gamma\,\mathrm{cap}_p\left(\overline{B_r};B_{2\,r}\right)}{\mathrm{cap}_p\Big(\overline{B_{r_1}};B_{2\,r}\Big)}\right)^\frac{1}{p}.
\]
By using this estimate in \eqref{dai!}, we obtain
\[
\lambda_p(\Omega)\,\left(1-\left(\frac{\delta+\gamma\,\mathrm{cap}_p\left(\overline{B_r};B_{2\,r}\right)}{\mathrm{cap}_p\Big(\overline{B_{r_1}};B_{2\,r}\Big)}\right)^\frac{1}{p}\right)^p\le \frac{2^{p-1}}{r^p}\,\frac{r^N}{|S_{r_1,r_2}|}\,\left[\frac{N\,\omega_N}{\varepsilon^{p-1}}+\delta+ \gamma\,\mathrm{cap}_p\Big(\overline{B_1};B_{2}\Big)\right].
\]
This is valid for every $\delta>0$, thus we can eliminate it by taking the limit as $\delta$ goes to $0$. We thus obtain
\begin{equation}
\label{quasifatto}
\lambda_p(\Omega)\,\left(1-\left(\frac{\mathrm{cap}_p\left(\overline{B_r};B_{2\,r}\right)}{\mathrm{cap}_p\Big(\overline{B_{r_1}};B_{2\,r}\Big)}\right)^\frac{1}{p}\,\gamma^\frac{1}{p}\right)^p\le \frac{2^{p-1}}{r^p}\,\frac{r^N}{|S_{r_1,r_2}|}\,\left[\frac{N\,\omega_N}{\varepsilon^{p-1}}+\gamma\,\mathrm{cap}_p\Big(\overline{B_1};B_{2}\Big)\right].
\end{equation}
Now, we make the choice $\ell=2\,\varepsilon$, so that 
\[
r_1=(1-2\,\varepsilon)\,r,\qquad r_2=\left(1-\varepsilon\right)\,r,
\]
and observe that
\[
\left(\frac{\mathrm{cap}_p\left(\overline{B_r};B_{2\,r}\right)}{\mathrm{cap}_p\Big(\overline{B_{r_1}};B_{2\,r}\Big)}\right)^\frac{1}{p}>1\qquad \text{ and }\qquad \lim_{\varepsilon\to 0}\left(\frac{\mathrm{cap}_p\left(\overline{B_r};B_{2\,r}\right)}{\mathrm{cap}_p\Big(\overline{B_{r_1}};B_{2\,r}\Big)}\right)^\frac{1}{p}=1.
\]
Thanks to the presence of the factor $\gamma<1$, this implies that up to choosing $\varepsilon>0$ small enough (depending on $\gamma$), we can uniformly bound from below the factor  multiplying $\lambda_p(\Omega)$ in \eqref{quasifatto}. Of course, the smaller we will choose $\varepsilon$, the larger the right-hand side of \eqref{quasifatto} will be (because of the factor $|S_{r_1,r_2}|$).
\par
In particular, we want to choose $0<\varepsilon<1/2$ so that 
\begin{equation} 
\label{scelta_eps_subconformal}
	1-\left(\frac{\mathrm{cap}_p\left(\overline{B_r};B_{2\,r}\right)}{\mathrm{cap}_p\Big(\overline{B_{r_1}};B_{2\,r}\Big)}\right)^\frac{1}{p}\,\gamma^\frac{1}{p}\ge \frac{1-\gamma^\frac{1}{p}}{2}.
\end{equation}
This is the same as
\begin{equation} 
\label{scelta_eps_subconformal2}
\left(\frac{1+\gamma^\frac{1}{p}}{2\,\gamma^\frac{1}{p}}\right)^\frac{p}{p-1}\ge \frac{1}{(1-2\,\varepsilon)^\frac{N-p}{p-1}}\, \left(\frac{\mathrm{cap}_p\left(\overline{B_1};B_{2}\right)}{\mathrm{cap}_p\Big(\overline{B_1};B_{2/(1-2\,\varepsilon)}\Big)}\right)^\frac{1}{p-1},
\end{equation}
where we also used Remark \ref{oss:scaling}.
We now need to further distinguish the case $p<N$ and $p=N$.
\vskip.2cm\noindent
{\it 1.A Case $1 < p <N$.}
By recalling \eqref{cappalla}, the condition \eqref{scelta_eps_subconformal2} is equivalent to
\[
\frac{1}{(1-2\,\varepsilon)^\frac{N-p}{p-1}}\,\frac{1}{1-\left(\dfrac{1}{2}\right)^\frac{N-p}{p-1}}\,\left(1-\left(\dfrac{1-2\,\varepsilon}{2}\right)^\frac{N-p}{p-1}\right)\le \left(\frac{1+\gamma^\frac{1}{p}}{2\,\gamma^\frac{1}{p}}\right)^\frac{p}{p-1}.
\]
By simplifying a bit the expression, this is equivalent to
\[
\frac{2^\frac{N-p}{p-1}}{2^\frac{N-p}{p-1}-1}\,\frac{1}{(1-2\,\varepsilon)^\frac{N-p}{p-1}}-\frac{1}{2^\frac{N-p}{p-1}-1}\le \left(\frac{1+\gamma^\frac{1}{p}}{2\,\gamma^\frac{1}{p}}\right)^\frac{p}{p-1}.
\]
In turn, this can be recast as follows
\[
(1-2\,\varepsilon)^\frac{N-p}{p-1}\,\left[\left(\frac{1+\gamma^\frac{1}{p}}{2\,\gamma^\frac{1}{p}}\right)^\frac{p}{p-1}+\frac{1}{2^\frac{N-p}{p-1}-1}\right]\ge \frac{2^\frac{N-p}{p-1}}{2^\frac{N-p}{p-1}-1}.
\]
After some simple (yet tedious) computations, we get that it is sufficient to choose
\begin{equation}
\label{epsilon0}
\varepsilon\le \varepsilon_0:=\min\left\{\frac{1}{4\,(N-1)},\, \frac{1}{2}-\left[\left(2^\frac{N-p}{p-1}-1\right)\,\left(\frac{1+\gamma^\frac{1}{p}}{2\,\gamma^\frac{1}{p}}\right)^\frac{p}{p-1}+1\right]^{-\frac{p-1}{N-p}}\right\}.
\end{equation}
We observe that $\varepsilon_0\le 1/4$, in particular.
Finally, by making this choice for $\varepsilon$, we get from \eqref{quasifatto} and \eqref{scelta_eps_subconformal}
\[
\lambda_p(\Omega)\le \frac{1}{r^p}\,\left(\frac{2}{1-\gamma^\frac{1}{p}}\right)^p\,\frac{2^{p-1}}{\omega_N\,\Big((1-\varepsilon_0)^N-(1-2\,\varepsilon_0)^N\Big)}\,\left[\frac{N\,\omega_N}{\varepsilon_0^{p-1}}+\gamma\,\mathrm{cap}_p\Big(\overline{B_1};B_{2}\Big)\right].
\]
We eventually get the claimed estimate \eqref{volio} with
\[
C_{N,p,\gamma}=\left(\frac{2}{1-\gamma^\frac{1}{p}}\right)^p\,2^p \left[\frac{1}{\varepsilon_0^p}+\frac{\gamma}{\varepsilon_0}\,\frac{\mathrm{cap}_p\Big(\overline{B_1};B_{2}\Big)}{N\,\omega_N}\right],
\]
once observed that\footnote{For the first inequality, use that 
\[
(1-\varepsilon_0)^N-(1-2\,\varepsilon_0)^N=N\,\int_{1-2\,\varepsilon_0}^{1-\varepsilon_0} \tau^{N-1}\,d\tau\ge N\,(1-2\,\varepsilon_0)^{N-1}\,\varepsilon_0.
\]
In the second inequality, use Bernoulli's inequality and the fact that $\varepsilon_0\le 1/(4\,(N-1))$.} 
\[
(1-\varepsilon_0)^N- (1-2\,\varepsilon_0)^N\ge N\,(1-2\,\varepsilon_0)^{N-1}\,\varepsilon_0\ge \frac{N}{2}\,\varepsilon_0.
\]
\vskip.2cm\noindent
{\it 1.B Case $p=N$.}
By recalling \eqref{cappallaconf}, the first condition \eqref{scelta_eps_subconformal2} is equivalent to
\[
\log \dfrac{2}{1-\varepsilon}\le \left(\frac{1+\gamma^\frac{1}{N}}{2\,\gamma^\frac{1}{N}}\right)^\frac{N}{N-1}\,\log 2.
\]
This is equivalent to
\[
\log (1-\varepsilon)\ge \left[1-\left(\frac{1+\gamma^\frac{1}{N}}{2\,\gamma^\frac{1}{N}}\right)^\frac{N}{N-1}\right]\,\log 2.
\]
By exponentiating, we obtain
\[
1-\varepsilon\ge 2^{1-\alpha_{N,\gamma}},\qquad \text{with}\ \alpha_{N,\gamma}:=\left(\frac{1+\gamma^\frac{1}{N}}{2\,\gamma^\frac{1}{N}}\right)^\frac{N}{N-1}.
\] 
If we choose
\begin{equation}
\label{epsilon0N}
\varepsilon\le \varepsilon_0:=\min\left\{1-2^{1-\alpha_{N,\gamma}},\frac{1}{4\,(N-1)}\right\},
\end{equation}
we then obtain the desired property. The conclusion now follows as before.
\vskip.2cm\noindent
{\bf 2. Case $p=1$.} We go back to \eqref{dai!}. 
As done before, by combining \eqref{torsiontype21} with assumptions \eqref{quasiottima} and \eqref{ipotesiF}, we infer that 
\[
\begin{split}
	\fint_{S_{r_1,r_2}} \varphi_\delta \,dx &\leq \frac{1}{\mathrm{cap}_1(\overline{B_{r_2}}; B_{2r})}\, \int_{B_{2r}} |\nabla \varphi_\delta|\,dx \leq  \frac{\delta+\gamma\,\mathrm{cap}_1(\overline{B_{r}}; B_{2r})}{\mathrm{cap}_1(\overline{B_{r_2}}; B_{2r})}.
\end{split}
\]
We use this upper bound in the left-hand side of \eqref{dai!}. This yields
\[
\lambda_1(\Omega)\,\left(1-\frac{\delta+\gamma\,\mathrm{cap}_1(\overline{B_{r}}; B_{2r})}{\mathrm{cap}_1(\overline{B_{r_2}}; B_{2r})}\right)\le \frac{1}{r}\,\frac{r^N}{|S_{r_1,r_2}|}\,\left[N\,\omega_N+\delta +\gamma\,\mathrm{cap}_1\Big(\overline{B_1};B_{2}\Big)\right].
\]
We remove again the useless parameter $\delta$, by taking the limit as this goes to $0$. We then obtain 
\[
\lambda_1(\Omega)\,\left(1-\frac{\gamma\,\mathrm{cap}_1(\overline{B_{r}}; B_{2r})}{\mathrm{cap}_1(\overline{B_{r_2}}; B_{2r})}\right)\le \frac{1}{r}\,\frac{r^N}{|S_{r_1,r_2}|}\,\left[N\,\omega_N+\gamma\,\mathrm{cap}_1\Big(\overline{B_1};B_{2}\Big)\right]. 
\]
The choice of the parameter $\varepsilon$ and $\ell$ is now simpler: observe in particular that the role of parameter $\ell$ {\it is now immaterial}, thanks to the fact that the left-hand side in the previous estimate {\it only depends on $r_2$, and not on $r_1$}. We can thus take the limit as $\ell$ goes to $1$ (that is, $r_1$ goes to $0$) and obtain
\begin{equation}
\label{quasifatto1}
\lambda_1(\Omega)\,\left(1-\frac{\gamma\,\mathrm{cap}_1(\overline{B_{r}}; B_{2r})}{\mathrm{cap}_1(\overline{B_{r_2}}; B_{2r})}\right)\le \frac{1}{r}\,\frac{r^N}{|B_{r_2}|}\,\left[N\,\omega_N+\gamma\,\mathrm{cap}_1\Big(\overline{B_1};B_{2}\Big)\right].
\end{equation}
Finally, we choose $\varepsilon>0$ in such a way that
\[
1-\frac{\mathrm{cap}_1\left(\overline{B_r};B_{2\,r}\right)}{\mathrm{cap}_1\Big(\overline{B_{r_2}};B_{2\,r}\Big)}\,\gamma\ge \frac{1-\gamma}{2}.
\]
This is the same as
\[
\mathrm{cap}_1(\overline{B_{r_2}};B_{2\,r})\ge \mathrm{cap}_1(\overline{B_{r}};B_{2\,r})\,\frac{2\,\gamma}{1+\gamma}.
\]
By recalling the expression \eqref{cappalla0}, we want
\[
(1-\varepsilon)^{N-1}\ge \frac{2\,\gamma}{1+\gamma}.
\]
Thus, by taking
\[
\varepsilon_0=\min\left\{1-\left(\frac{2\,\gamma}{1+\gamma}\right)^\frac{1}{N-1},\frac{1}{2\,N}\right\},
\]
we get from \eqref{quasifatto1}
\[
\lambda_1(\Omega)\le 
\frac{1}{r}\,\frac{2}{1-\gamma}\,\frac{2}{\omega_N}\,\left[N\,\omega_N+\gamma\,\mathrm{cap}_1\Big(\overline{B_1};B_{2}\Big)\right].
\]
Observe that we also used that 
\[
|B_{r_2}|=\omega_N\,(1-\varepsilon_0)^N\,r^N\ge \omega_N\,(1-N\,\varepsilon_0)\,r^N\ge \frac{\omega_N}{2}\,r^N,
\]
thanks to the choice of $\varepsilon_0$. Thus, we get \eqref{volio}, as desired.
\end{proof}

\begin{remark}[Quality of the constant]
\label{rem:constant}
We discuss the asymptotic behaviour of the constant $C_{N,p,\gamma}$ obtained in the previous result, as $\gamma$ goes to $1$. We distinguish two cases.
\begin{itemize}
	\item {\it Case $p=1$}: this is easier, in this case we have obtained
	 \begin{equation*}
		C_{N, 1, \gamma} = 4\,N\,\frac{1+\gamma}{1-\gamma},
	\end{equation*}
	where we also used the explicit expression of the relative $1-$capacity of $\overline{B_1}$.
	Hence, we have the following asymptotic behaviour
	\[
	0<\lim_{\gamma\nearrow 1} (1-\gamma)\,C_{N,1,\gamma}<+\infty.
	\]
\item {\it Case $1<p\le N$}: we first observe that as $\gamma$ goes to $1$, from \eqref{epsilon0} and \eqref{epsilon0N} we have
		 \begin{equation*}
			\varepsilon_0 = \left\{\begin{array}{ll}
				(1-\gamma)\,\dfrac{1}{N-p}\,\dfrac{2^\frac{N-p}{p-1}-1}{2\cdot 2^\frac{N-1}{p-1}}+o(1-\gamma), &\text{ if } 1 < p <N,\\ 
				&\\
				(1 - \gamma)\,\dfrac{1}{N-1}\,\dfrac{\log2}{2} + o(1-\gamma), &\text{ if } p=N.
			\end{array}
			\right. 
		\end{equation*}
Thus, by inspecting the proof above, we have 
\begin{equation}
0<\lim_{\gamma\nearrow 1} (1-\gamma)^{2\,p}\,C_{N,p,\gamma}<+\infty.
	\end{equation}	
\end{itemize}
We see in particular that the constant $C_{N,p,\gamma}$ blows-up as $\gamma$ goes to $1$. While not claiming that the behavior above is optimal, we point out that the upper bound of Theorem \ref{teo:upper_bound} can not be true with a constant which stays finite as $\gamma$ goes to $1$. We refer to Example \ref{exa:poliquin} for a counter-example.
\end{remark}

\section{Extension to Poincar\'e-Sobolev embedding constants}
\label{sec:6}

In this section, we briefly discuss how the Main Theorem can be extended to the more general case of the
 {\it sharp Poincar\'e-Sobolev embedding constants} associated to an open set $\Omega$. More precisely, given $1 \leq p \le N$ and a {\it subcritical} exponent $q \geq 1$, that is 
 \begin{equation} 
 \label{sottocritico}
	\begin{cases}
		q < p^{*}, \quad &\text{ if } 1 \leq p < N,\\
		q < \infty, \quad &\text{ if } p=N,\\
	\end{cases}
\end{equation} 
where $p^{*}$ is the Sobolev conjugate exponent of $p$, we introduce the following quantity 
\begin{equation}
	\lambda_{p,q}(\Omega) = \inf_{\varphi\in C^\infty_0(\Omega)} \left\{\int_\Omega |\nabla u|^p\,dx\, :\, \|u\|^p_{L^q(\Omega)}=1\right\},
\end{equation}
sometimes referred to as {\it generalized principal frequency of the Dirichlet $p$-Laplacian}.
\par
With some minor modifications of the proofs of Theorems \ref{teo:lower_bound} and \ref{teo:upper_bound}, we have a two-sided estimate in terms of the capacitary inradius $R_{p,\gamma}(\Omega)$, for these quantities as well. We point out that for the lower bound {\it the additional restriction $q\ge p$ is mandatory} (see Remark \ref{oss:subbuteo} below).
\begin{theorem}
	Let  $1 \leq p \le N$ and let $q\ge 1$ satisfy \eqref{sottocritico}. Let  $0 < \gamma < 1$ and let $\Omega \subseteq \mathbb{R}^N$ be an open set. Then, we have 
	\begin{equation} 
	\label{lambda_pq_upper}
		\lambda_{p,q}(\Omega) \leq C_{N, \gamma, p, q} \left(\frac{1}{R_{p,\gamma}(\Omega)}\right)^{p-N+N\frac{p}{q}},
	\end{equation}
where it is intended that $\lambda_{p,q}(\Omega)=0$, whenever $R_{p,\gamma}(\Omega)=+\infty$.
\par	
	Furthermore, if $q \geq p$ we also have 
	\begin{equation} 
	\label{lambda_pq_lower}
		\gamma\,\sigma_{N, p, q} \left(\frac{1}{R_{p,\gamma}(\Omega)}\right)^{p-N+N\frac{p}{q}} \leq \lambda_{p,q}(\Omega).
	\end{equation} 
\end{theorem}

\begin{proof}
We prove \eqref{lambda_pq_upper} and \eqref{lambda_pq_lower} separately.
	 \vskip.2cm\noindent
	{\bf Upper bound.} The proof goes along the same lines as before. In particular, by using the same notation as in the proof of Theorem \ref{teo:upper_bound}, we now get 
	\[
		\lambda_{p,q}(\Omega) \left(\fint_{S_{r_1, r_2}} (1 - \varphi_\delta)^q\,dx\right)^{\frac{p}{q}} \leq \frac{2^{p-1}}{r^{p-N+N\frac{p}{q}}}\,\left(\frac{r^N}{|S_{r_1, r_2}|}\right)^{\frac{p}{q}}\left[\frac{N \omega_N}{\varepsilon^{p-1}} + \delta + \gamma\, \mathrm{cap}_p(\overline{B}_r; B_{2r}) \right],
	\]
	in place of \eqref{passo1}.
	We use Jensen's inequality\footnote{For $q=1$ this is not needed.} to estimate from below the leftmost term. This gives 
	\begin{equation} \label{passo2_lambda_pq}
		\lambda_{p,q}(\Omega) \left(1 - \fint_{S_{r_1, r_2}} \varphi_\delta\, dx\right)^p \leq \frac{2^{p-1}}{r^{p-N+N\frac{p}{q}}}\,\left(\frac{r^N}{|S_{r_1, r_2}|}\right)^{\frac{p}{q}} \left[\frac{N \omega_N}{\varepsilon^{p-1}} + \delta + \gamma\, \mathrm{cap}_p(\overline{B}_r; B_{2r}) \right],
	\end{equation}
	in place of \eqref{dai!}. We distinguish again the case $1<p\le N$ and the case $p=1$.
	\vskip.2cm\noindent
	{\it A. Case $1< p \le N $.}  As done before, by applying \eqref{torsiontype3} with $R = 2r$ and then taking the limit as $\delta$ goes to $0$, we obtain 
	\begin{equation*}
		\lambda_{p,q}(\Omega) \left(1 - \gamma^\frac{1}{p}\,\left( \dfrac{\mathrm{cap}_p(\overline{B_r}; B_{2r})}{\mathrm{cap}_p(\overline{B_{r_1}}; B_{2r})}\right)^\frac{1}{p}\right)^p \leq \frac{2^{p-1}}{r^{p-N+N\frac{p}{q}}}\,\left(\frac{r^N}{|S_{r_1, r_2}|}\right)^{\frac{p}{q}} \left[\frac{N \omega_N}{\varepsilon^{p-1}} + \gamma\, \mathrm{cap}_p(\overline{B_1}; B_{2}) \right].
	\end{equation*}
Observe that this is the same as \eqref{quasifatto}, except for the presence of the correct scaling power on $r$ and the power $p/q$ on the term $r^N/|S_{r_1, r_2}|$, in the right-hand side. Then one concludes as in the case $p=q$ previously treated.	
	\vskip.2cm \noindent
	{\it B. Case $p=1$.} In \eqref{passo2_lambda_pq}, we use this time \eqref{torsiontype21}. By taking the limit as $\delta$ goes to $0$ again, we infer that 
\[
\lambda_{1,q}(\Omega)\,\left(1-\frac{\gamma\,\mathrm{cap}_1(\overline{B_{r}}; B_{2r})}{\mathrm{cap}_1(\overline{B_{r_2}}; B_{2r})}\right)\le \frac{1}{r^{1-N+\frac{N}{q}}}\,\left(\frac{r^N}{|S_{r_1,r_2}|}\right)^\frac{1}{q}\,\left[N\,\omega_N+\gamma\,\mathrm{cap}_1\Big(\overline{B_1};B_{2}\Big)\right]. 
\]
We can take again the limit as $r_1$ goes to $0$ and obtain
\[
\lambda_{1,q}(\Omega)\,\left(1-\frac{\gamma\,\mathrm{cap}_1(\overline{B_{r}}; B_{2r})}{\mathrm{cap}_1(\overline{B_{r_2}}; B_{2r})}\right)\le  \frac{1}{r^{1-N+\frac{N}{q}}}\,\left(\frac{r^N}{|S_{r_1,r_2}|}\right)^\frac{1}{q}\,\left[N\,\omega_N+\gamma\,\mathrm{cap}_1\Big(\overline{B_1};B_{2}\Big)\right],
\]
in place of \eqref{quasifatto1}. The conclusion then follows as in the case $q=p=1$.
	\vskip.2cm\noindent
	{\bf Lower bound.}
	We can assume $R_{p,\gamma}(\Omega) < +\infty$, otherwise there is nothing to prove.	
	 Let $r > R_{p,\gamma}(\Omega)$ and let $u \in C^\infty_0(\Omega)$. As in the proof of Theorem \ref{teo:lower_bound}, for $p<q$ satisfying \eqref{sottocritico}, we can still apply the Maz'ya-Poincar\'e inequality \cite[Theorem 2.5]{BozBra} and get this time 
	\begin{equation*}
		\frac{\mathscr{C}}{r^{\frac{N}{q}}}\, \mathrm{cap}_p\left(\overline{B_r(x_0)} \setminus \Omega; B_{2\sqrt{N}r}(x_0)\right)^{\frac{1}{p}}\, \|u\|_{L^q(Q_{r}(x_0))} \leq \|\nabla u\|_{L^p(Q_{r}(x_0))},
	\end{equation*}
	where $\mathscr{C}$ is the same constant as in \cite[Theorem 2.5]{BozBra}. Observe that now it depends on $q$, as well. The relative $p-$capacity on the left-hand side can be estimated from below as in the proof of Theorem \ref{teo:lower_bound}, so to get
	\[ 
		\frac{\mathscr{C}^p\,\mathrm{cap}_p(\overline{B_1}; B_2)}{\dfrac{2\,\sqrt{N}}{\lambda_p(B_1)^{\frac{1}{p}}}+1}\, \left(\frac{1}{r}\right)^{p-N+p\, \frac{N}{q}}\,\gamma\, \|u\|^p_{L^q(Q_{r}(x_0))} \leq \|\nabla u\|^p_{L^p(Q_{r}(x_0))},
	\]
in place of \eqref{diomadonna}. In order to conclude, we want to use again a tiling of $\mathbb{R}^N$, made of a countable family of disjoint cubes with radius $r$. A slight difference now arises, which explains the restriction on $q$: indeed,  
if $\{\mathcal{Q}_{\alpha}\}_{\alpha \in \mathbb{N}}$ is such a family of cubes, we have this time
\[
\sum_{\alpha\in\mathbb{N}}\|\nabla u\|^p_{L^p(\mathcal{Q}_{\alpha})}=\|\nabla u\|_{L^p(\Omega)}^p\qquad \text{but}\qquad \sum_{\alpha\in\mathbb{N}} \|u\|^p_{L^q(\mathcal{Q}_\alpha)}\not= \|u\|^p_{L^q(\Omega)}.
\]
However, the choice $q>p$ entails that the function $\delta \mapsto \delta^{p/q}$ is subadditive. Thus, in particular 
	\[
	 \sum_{\alpha\in\mathbb{N}} \|u\|^p_{L^q(\mathcal{Q}_\alpha)} \geq \left(\sum_{\alpha\in\mathbb{N}} \|u\|^q_{L^q(\mathcal{Q}_\alpha)}\right)^\frac{p}{q} = \|u\|^p_{L^q(\Omega)}.
	\]
	We can now get the desired conclusion, as in the case $q=p$.
\end{proof}

\begin{remark}
\label{oss:subbuteo}
The previous proof for the lower bound {\it does not} work if $1 \leq q < p$. This is not by chance:
in Example \ref{exa:app_lower_bound_fail} we construct a counter-example to the validity of the lower bound, in this case.
\end{remark}

\section{A comment on the case $p>N$}
\label{sec:7}
In the previous sections we excluded the range $p>N$, since in this case we already know that $\lambda_p(\Omega)$ admits a two-sided estimate in terms of the usual inradius $r_\Omega$, thanks to \eqref{1intro} and \eqref{2intro}. Actually, this remains true more generally for $\lambda_{p,q}(\Omega)$ and $p\le q\le \infty$, see \cite[Corollary 5.9]{BraPriZag2} (and also \cite[Corollary 4.7]{BozBra}).
\par
In this section, for completeness we compare the capacitary inradius and the usual notion of inradius when $p>N$. We will show that they coincide, at least for $\gamma$ smaller than a certain (optimal) threshold. This result is certainly not surprising, but it requires some work and some precise estimates on the capacity of points. 
\par
At this aim, we start by pointing out that for a compact set $\Sigma\Subset B_R$, the definition of relative $p-$capacity can be also written as
\[
\mathrm{cap}_p(\Sigma;B_R)=\inf_{\varphi\in W^{1,p}_0(B_R)}\left\{\int_{B_R} |\nabla \varphi|^p\,dx\, :\, \varphi \ge 1 \ \text{on}\ \Sigma\right\},
\]
for $p>N$. 
Observe that the pointwise requirement on the test functions make sense, in light of Morrey's inequality, i.e. $W^{1,p}_0(B_R)$ is embedded in a space of continuous functions on $\overline{B_R}$. Moreover, by a standard application of the Direct Method, the previous infimum is actually (uniquely) attained, by a function $u_\Sigma$ called $p-$capacitary potential. By minimality and uniqueness, it is not difficult to see that this is a $p-$harmonic function in $B_R\setminus \Sigma$, such that 
\[
0\le u_\Sigma\le 1\qquad \text{and}\qquad u_\Sigma=1 \text{ on } \Sigma.
\]
\begin{lemma}
\label{lm:points}
Let $p>N\ge 2$ and let $R>0$. We choose a set of distinct points $\{x_1,\dots,x_k\}\in B_R$ and set
\[
D:=\min\Big\{|x_i-x_j|,\, \mathrm{dist}(x_i;\partial B_R)\, :\, i,j\in\{1,\dots,k\},\, i\not=j\Big\}>0.
\]
There exists a constant $c_p>0$, depending on $p$ only, such that that for every $\delta<D$ we have
\[
\mathrm{cap}_p\big(\{x_1,\dots,x_{k-1}\};B_R\big)+c_p\,\int_{B_\delta(x_k)} |\nabla u-\nabla H_u|^p\,dx\le \mathrm{cap}_p\big(\{x_1,\dots,x_{k}\};B_R\big).
\]
Here $u$ is the $p-$capacitary potential of the set $\{x_1,\dots,x_k\}$ relative to $B_R$, while $H_u$ is the $p-$harmonic function in $B_\delta(x_k)$ such that $u-H_u\in W^{1,p}_0(B_\delta(x_k))$. In particular, we have
\[
\mathrm{cap}_p\big(\{x_1,\dots,x_{k-1}\};B_R\big)< \mathrm{cap}_p\big(\{x_1,\dots,x_{k}\};B_R\big).
\]
\end{lemma}
\begin{proof}
We take $u\in W^{1,p}_0(B_R)$ to be an optimal function for $\mathrm{cap}_p\big(\{x_1,\dots,x_{k}\};B_R\big)$. This means that $0\le u\le 1$ and
\[
\int_{B_R} |\nabla u|^p\,dx=\mathrm{cap}_p\big(\{x_1,\dots,x_{k}\};B_R\big),\qquad u(x_i)=1,\ \text{for}\ i=1,\dots,k. 
\]
Observe that by minimality, the function $u$ is weakly $p-$harmonic in the open connected set $B_R\setminus\{x_1,\dots,x_k\}$. Thus, by the minimum and maximum principles, we get that
\[
0<u(x)<1\qquad \text{ in }B_R\setminus\{x_1,\dots,x_k\}.
\]
We will use a ``$p-$harmonic replacement trick'' in order to modify $u$ and produce a trial function, which is admissible for the $p-$capacity of $\{x_1,\dots,x_{k-1}\}$. Namely, we introduce the new function
\[
U(x)=\left\{\begin{array}{rl}
u(x), & \text{if}\ x\in B_R\setminus B_\delta(x_k),\\
H_u(x),& \text{if}\ x\in B_\delta(x_k),
\end{array}
\right.
\]
where $H_u\in W^{1,p}(B_\delta(x_k))$ is the unique minimizer of 
\[
\min_{\varphi\in W^{1,p}(B_\delta(x_k))} \left\{\int_{B_\delta(x_k)}|\nabla \varphi|^p\,dx\, :\, \varphi-u\in W^{1,p}_0(B_\delta(x_k))\right\}.
\]
Observe that by minimality, the function $H_u$ satisfies
\[
\int_{B_\delta(x_k)} \langle |\nabla H_u|^{p-2}\,\nabla H_u,\nabla \varphi\rangle\,dx=0,\qquad \text{for every}\ \varphi\in W^{1,p}_0(B_\delta(x_k)).
\]
Thus, in particular, we have 
\begin{equation}
\label{H1}
\int_{B_\delta(x_k)} \langle |\nabla H_u|^{p-2}\,\nabla H_u,\nabla u-\nabla H_u\rangle\,dx=0.
\end{equation}
It is not difficult to see that the function $U$ is admissible for the $p-$capacity of $\{x_1,\dots,x_{k-1}\}$. This gives
\[
\begin{split}
\mathrm{cap}_p\big(\{x_1,\dots,x_{k-1}\};B_R\big)&\le \int_{B_R} |\nabla U|^p\,dx\\
&=\int_{B_R\setminus B_\delta(x_k)} |\nabla u|^p\,dx+\int_{B_\delta(x_k)} |\nabla H_u|^p\,dx\\
&=\int_{B_R} |\nabla u|^p\,dx+\left(\int_{B_\delta(x_k)} |\nabla H_u|^p\,dx-\int_{B_\delta(x_k)} |\nabla u|^p\,dx\right)\\
&=\mathrm{cap}_p\big(\{x_1,\dots,x_{k}\};B_R\big)+\left(\int_{B_\delta(x_k)} |\nabla H_u|^p\,dx-\int_{B_\delta(x_k)} |\nabla u|^p\,dx\right).
\end{split}
\]
In order to conclude, we just need to estimate the rightmost term into parentheses. To this aim, we need to recall the following convexity inequality, which is valid for $p>2$ (see \cite[Lemma 4.2, equation (4.3)]{Lin}):
\[
|z|^p\ge |w|^p+p\,\langle |w|^{p-2}\,w,z-w\rangle+c_p\,|z-w|^p,\qquad \text{for every}\ z,w\in\mathbb{R}^N.
\]
From this inequality, we get
\[
\begin{split}
\int_{B_\delta(x_k)}|\nabla u|^p\,dx&\ge \int_{B_\delta(x_k)}|\nabla H_u|^p\,dx+\int_{B_\delta(x_k)} \langle |\nabla H_u|^{p-2}\,\nabla H_u,\nabla u-\nabla H_u\rangle\,dx\\
&+c_p\,\int_{B_\delta(x_k)} |\nabla u-\nabla H_u|^p\,dx\\
&=\int_{B_\delta(x_k)}|\nabla H_u|^p\,dx+c_p\,\int_{B_\delta(x_k)} |\nabla u-\nabla H_u|^p\,dx.
\end{split}
\]
In the last identity, we used \eqref{H1}. This implies that we have 
\[
\mathrm{cap}_p\big(\{x_1,\dots,x_{k-1}\};B_R\big)\le \mathrm{cap}_p\big(\{x_1,\dots,x_{k}\};B_R\big)-c_p\,\int_{B_\delta(x_k)} |\nabla u-\nabla H_u|^p\,dx.
\]
Finally, we observe that the last quantity can not vanish, otherwise $u$ would be weakly $p-$harmonic on $B_\delta(x_k)$ and would attain its maximum at the center of the ball, thus violating the maximum principle.
\end{proof}
\begin{lemma}
\label{lm:punto}
Let $p>N$, for every $x_0\in B_R(y_0)$, we have
\[
\mathrm{cap}_p\left(\{x_0\};B_{R}(y_0)\right)\ge \mathrm{cap}_p\left(\{y_0\};B_{R}(y_0)\right).
\]
\end{lemma}
\begin{proof}
We can suppose that $y_0$ coincides with the origin.
It is sufficient to use \cite[(2.2.10)]{Maz} with $F=\{x_0\}$. This gives
\[
\mathrm{cap}_p\left(\{x_0\};B_{R}\right)\ge (N\,\omega_N)^\frac{p}{N}\,N^\frac{N-p}{N}\,\left(\frac{p-N}{p-1}\right)^{p-1}\,|B_R|^\frac{N-p}{N}.
\]
By recalling \eqref{cappunto}, we easily see that the right-hand side coincides with the capacity of the center of the ball. 
\end{proof}
In the next result we compare the usual inradius with the capacitary one, in the superconformal case $p>N$. 
We will show that for $\gamma$ smaller than a universal sharp constant, they actually coincide.
\begin{proposition}
Let $p>N$ and 
\begin{equation}
\label{gamma0}
\gamma_0:=\frac{\mathrm{cap}_p\left(\{0\};B_{2}\right)}{\mathrm{cap}_p\left(\overline{B_1};B_{2}\right)}=\frac{1}{2^{p-N}}\,\left(2^\frac{p-N}{p-1}-1\right)^{p-1}.
\end{equation}
For every open set $\Omega\subseteq\mathbb{R}^N$ we have 
\[
R_{p,\gamma}(\Omega)=r_\Omega,\qquad \text{for every}\ 0\le \gamma<\gamma_0.
\]
Moreover, for the punctured ball $\dot B_R:=B_R\setminus\{0\}$ we have 
\[
R_{p,\gamma}(\dot B_R)>r_{\dot B_R},\qquad \text{for every}\ 1>\gamma\ge \gamma_0.
\]
\end{proposition}
\begin{proof}
We prove the two facts separately.
\vskip.2cm\noindent
We have already observed that 
\[
r_\Omega\le R_{p,\gamma}(\Omega).
\]
In particular, if $r_\Omega=+\infty$, then the conclusion trivially follows. Let us suppose that $r_\Omega<+\infty$. For every ball $B_r(y_0)$ with $r>r_\Omega$, we then must have 
\[
\overline{B_r(y_0)}\setminus \Omega\not=\emptyset.
\] 
In particular, there exists a point $x_0\in \overline{B_r(y_0)}\setminus \Omega$. By monotonicity of the $p-$capacity with respect to the set inclusion, we get
\[
\mathrm{cap}_p\left(\overline{B_r(y_0)}\setminus \Omega;B_{2\,r}(y_0)\right)\ge \mathrm{cap}_p\left(\{x_0\};B_{2\,r}(y_0)\right)\ge \mathrm{cap}_p\left(\{y_0\};B_{2\,r}(y_0)\right).
\]
In the second inequality, we used Lemma \ref{lm:punto}. In particular, by recalling the definition of $\gamma_0$, we get
\[
\mathrm{cap}_p\left(\overline{B_r(y_0)}\setminus \Omega;B_{2\,r}(y_0)\right)\ge \gamma_0\,r^{N-p}\,\mathrm{cap}_p\left(\overline{B_1};B_{2}\right)=\gamma_0\,\mathrm{cap}_p\left(\overline{B_r(y_0)};B_{2\,r}(y_0)\right).
\]
This implies that if $\gamma<\gamma_0$, then $\overline{B_r(y_0)}\setminus \Omega$ is not $(p,\gamma)-$negligible, for $r>r_\Omega$. This gives the desired conclusion, in light of the definition of $R_{p,\gamma}(\Omega)$.
\vskip.2cm\noindent
We now show the optimality of the previous result.
We consider the punctured ball $\dot B_R=B_R\setminus\{0\}$. We clearly have $r_{\dot B_R}=R/2$. On the other hand, it is not difficult to show that 
\[
R_{p,\gamma}(\dot B_R)\ge R,\qquad \text{for every}\ \gamma\ge\gamma_0,
\]
where $\gamma_0$ is still defined by \eqref{gamma0}. Indeed, we may notice that if $r<R$ we have
\[
\mathrm{cap}_p(\overline{B_r}\setminus \dot B_R;B_{2r})=\mathrm{cap}_p(\{0\};B_{2r})=r^{N-p}\,\gamma_0\,\mathrm{cap}_p(\overline{B_r};B_{2r}),
\]
which shows that $\overline{B_r}\setminus \dot B_R$ is $(p,\gamma_0)-$negligible. Thus, this already shows that for $\gamma\ge \gamma_0$
\[
R_{p,\gamma}(\dot B_R)\ge R_{p,\gamma_0}(\dot B_R)\ge r,\qquad \text{for every}\ r<R.
\]
Actually, we can show that $R_{p,\gamma_0}(\dot B_R)=R$. It is sufficient to observe that for every $r> R$ and every $x_0\in\mathbb{R}^N$, the set $\overline{B_r(x_0)}\setminus \dot B_R$ contains at least two distinct points. By Lemma \ref{lm:points}, this implies that 
\[
\begin{split}
\mathrm{cap}_p(\overline{B_r(x_0)}\setminus \dot B_R;B_{2r}(x_0))&>\mathrm{cap}_p(\{0\};B_{2r})\\
&=\gamma_0\,\mathrm{cap}_p(\overline{B_r};B_{2r})=\gamma_0\,\mathrm{cap}_p(\overline{B_r(x_0)};B_{2r}(x_0)),
\end{split}
\] 
that is any ball with radius $r> R$ is not $(p,\gamma_0)-$negligible. This gives $R_{p,\gamma_0}(\dot B_R)=R$, as claimed.
\end{proof}

\appendix 

\section{Some counter-examples}
\label{sec:A}

\begin{example}[Failure for $\gamma=0$]
\label{exa:gamma0}
For $1\le p\le N$ and an open set $\Omega\subseteq\mathbb{R}^N$, we introduce the quantity
\[
\mathfrak{R}_{\Omega}:=R_{p,0}(\Omega)=\sup\Big\{r>0\, :\, \exists x_0\in\mathbb{R}^N\ \text{such that}\ \mathrm{cap}_p\left(\overline{B_r(x_0)}\setminus\Omega;B_{2r}(x_0)\right)=0\Big\}.
\]
This may appear as the natural capacitary extension of the usual inradius. However, in this section, we will give an example showing that this notion is not strong enough to permit having the uniform lower bound
\[
\lambda_p(\Omega)\ge C\,\left(\frac{1}{\mathfrak{R}_{\Omega}}\right)^p,
\]
for every $\Omega\subseteq\mathbb{R}^N$ open set. Indeed, 
for every $0<\varepsilon<1/4$, we introduce the periodically perforated set
\[
\Omega_\varepsilon= \mathbb{R}^N \setminus \bigcup_{\mathbf{i}\in\mathbb{Z}^N} \overline{B_\varepsilon(\mathbf{i})}.
\] 
We claim that 
\begin{equation}
\label{gammazzero}
\lim_{\varepsilon\searrow 0} \lambda_p(\Omega_\varepsilon)=0\qquad \text{while}\qquad \mathfrak{R}_{\Omega_\varepsilon}\le \frac{\sqrt{N}}{2},\qquad \text{for every}\ 0<\varepsilon<\frac{1}{4}.
\end{equation}
We first observe that the usual inradius of $\Omega_\varepsilon$ is uniformly bounded, that is
\[
r_{\Omega_\varepsilon}\le \frac{\sqrt{N}}{2},\qquad \text{for every}\ 0<\varepsilon<\frac{1}{4}.
\]
In particular, for every ball $B_r(x_0)$ with $r>\sqrt{N}/2$ we have 
\[
B_r(x_0)\cap \left(\bigcup_{\mathbf{i}\in\mathbb{Z}^N} \overline{B_\varepsilon(\mathbf{i})}\right)\not=\emptyset.
\]
More precisely, let $\mathbf{i}_0\in\mathbb{Z}^N$ be such that 
\[
|x_0-\mathbf{i}_0|=\mathrm{dist}(x_0,\mathbb{Z}^N),
\]
this distance does not exceed $\sqrt{N}/2$. Consequently, we have $\mathbf{i}_0\in B_r(x_0)$ and thus
\[
|B_r(x_0)\setminus\Omega_\varepsilon|\ge |B_r(x_0)\cap \overline{B_\varepsilon(\mathbf{i}_0)}|>0.
\]
By the properties of capacity (see equations \cite[(2.2.10) \& (2.2.11) pag. 148]{Maz}), we can infer that 
\[
\mathrm{cap}_p\left(\overline{B_r(x_0)} \setminus \Omega_\varepsilon; B_{2r}(x_0)\right)>0.
\]
Thus, for every $r>\sqrt{N}/2$, we have that $\overline{B_r(x_0)} \setminus \Omega$ has positive relative $p-$capacity and, according to the definition, we obtain
\[
\mathfrak{R}_{\Omega_\varepsilon}\le \frac{\sqrt{N}}{2},\qquad \text{for every}\ 0<\varepsilon<\frac{1}{4},
\]
as well. In order to conclude, we need to prove the first property in \eqref{gammazzero}. It is not difficult to see that 
\[
\lambda_p(\Omega_\varepsilon)=\inf_{u\in \mathrm{Lip}(\overline{Q_{1/2}})} \left\{\int_{Q_{1/2}} |\nabla u|^p\,dx\, :\, \|u\|_{L^p(Q_{1/2})}=1,\, u=0\ \text{on}\ \overline{B_\varepsilon}\right\}.
\]
It is sufficient to proceed as in the proof of \cite[Lemma 4.4]{BozBra}, for example. In particular, we take $\varphi\in \mathrm{Lip}_0(B_{1/2})$ such that $\varphi=1$ on $\overline{B_{\varepsilon}}$ and $0\le \varphi\le 1$, extended by $0$ to $\overline{Q_{1/2}}\setminus B_{1/2}$. By using the test function $u=(1-\varphi)/\|1-\varphi\|_{L^q(Q_{1/2})}$, we get
\[
\lambda_p(\Omega_\varepsilon)\le \frac{\displaystyle\int_{B_{1/2}}|\nabla\varphi|^p\,dx}{\displaystyle \int_{Q_{1/2}}(1-\varphi)^p\,dx}\le \frac{\displaystyle\int_{B_{1/2}}|\nabla\varphi|^p\,dx}{|Q_{1/2}\setminus B_{1/2}|}.
\]
Thanks to the arbitrariness of $\varphi$ and recalling formula \eqref{capalt} from Remark \ref{oss:altrefunzioni}, we obtain
\[
\lambda_p(\Omega_\varepsilon)\le \frac{\mathrm{cap}_p(\overline{B_\varepsilon};B_{1/2})}{|Q_{1/2}\setminus B_{1/2}|}.
\]	
By using \eqref{cappalla0}, \eqref{cappalla} and \eqref{cappallaconf}, the previous estimate finally implies \eqref{gammazzero}.
\end{example}

\begin{example}[Failure of the lower bound for $q<p$]
\label{exa:app_lower_bound_fail}
We exhibit an open set $\Omega \subseteq \mathbb{R}^N$ such that for $1\le q<p\le N$
\[
\lambda_{p,q}(\Omega)=0\qquad \text{and}\qquad R_{p,\gamma}(\Omega)<+\infty,\ \text{for every}\ 0<\gamma<1.
\] 
This implies that the lower bound \eqref{lambda_pq_lower} cannot be true in this case.
We stick for simplicity to the case $1< p < N$, the case $p=N$ can be treated with minor modifications.  We take the slab 
\[
\Omega=\mathbb{R}^{N-1}\times(-1,1),
\] 
for which we have $\lambda_{p,q}(\Omega)=0$ for every $1\le q<p$ (see for example \cite[Proposition 6.1]{BrPini}). We need to prove that its capacitary inradius is finite, for every $0<\gamma<1$. At this aim, 
we fix $0<\gamma<1$ and take a ball $B_r(x_0)$ such that $r>1$ and 
\[
\mathrm{cap}_p(\overline{B_r(x_0)}\setminus\Omega;B_{2r}(x_0))\le \gamma\,\mathrm{cap}_p(\overline{B_r(x_0)};B_{2r}(x_0)).
\]
Thanks to the invariance of $\Omega$ by translations in directions belonging to $\{x_N=0\}$, we can suppose without loss of generality that $x_0=t\,\mathbf{e}_N$, for some $t\in\mathbb{R}$. Thus, we have
\begin{equation}
	\label{1}
	\mathrm{cap}_p(\overline{B_r(t\,\mathbf{e}_N)}\setminus\Omega;B_{2r}(t\,\mathbf{e}_N))\le \gamma\,r^{N-p}\,\mathrm{cap}_p(\overline{B_1};B_{2}),
\end{equation}
where we also used Remark \ref{oss:scaling}.
By using \cite[(2.2.10)]{Maz}, we get
\[
\begin{split}
\mathrm{cap}_p(\overline{B_r(t\,\mathbf{e}_N)}\setminus\Omega;B_{2r}(t\,\mathbf{e}_N))&\ge (N\,\omega_N)^\frac{p}{N}\,N^\frac{N-p}{N}\,\left(\frac{N-p}{p-1}\right)^{p-1}\\
&\times\left||B_{2r}(t\,\mathbf{e}_N)|^\frac{p-N}{N\,(p-1)}-|B_r(t\,\mathbf{e}_N)\setminus\Omega|^\frac{p-N}{N\,(p-1)}\right|^{1-p}.
\end{split}
\]
The expression on the right-hand side can be simplified: indeed, if we introduce $r^*$ the radius such that
\[
|B_{r^*}(t\,\mathbf{e}_N)|=|B_r(t\,\mathbf{e}_N)\setminus\Omega|\qquad \text{that is}\qquad r^*=\left(\frac{|B_r(t\,\mathbf{e}_N)\setminus\Omega|}{\omega_N}\right)^\frac{1}{N}.
\]
and recall \eqref{cappalla}, one can see that it coincides with
\[
\mathrm{cap}_p(\overline{B_{r^*}(t\,\mathbf{e}_N)};B_{2r}(t\,\mathbf{e}_N)).
\]
Accordingly, we obtain
\begin{equation}
	\label{2}
	\mathrm{cap}_p(\overline{B_r(t\,\mathbf{e}_N)}\setminus\Omega;B_{2r}(t\,\mathbf{e}_N))\ge \mathrm{cap}_p(\overline{B_{r^*}(t\,\mathbf{e}_N)};B_{2r}(t\,\mathbf{e}_N)).
\end{equation}
The volume of $B_r(t\,\mathbf{e}_N)\setminus\Omega$ can be uniformly bounded from below. Indeed, observe at first that if we set 
\[
\Omega_m=\prod_{i=1}^{N-1}(-m,m)\times(-1,1),
\] 
then
\[
\overline{B_r(t\,\mathbf{e}_N)}\setminus\Omega=\overline{B_r(t\,\mathbf{e}_N)}\setminus\Omega_m,\qquad \text{for every}\ m>r. 
\]
Then, as a consequence of \cite[Lemma 3.13]{BraMag}, for every $m>r$ the function
\[
t\mapsto |B_r(t\,\mathbf{e}_N)\setminus\Omega_m|=|B_r(t\,\mathbf{e}_N)|-|B_r(t\,\mathbf{e}_N)\cap\Omega_m|,
\]
attains its minimum for $t=0$ (see Figure \ref{fig:pallastriscia}).
\begin{figure}
	\includegraphics[scale=.3]{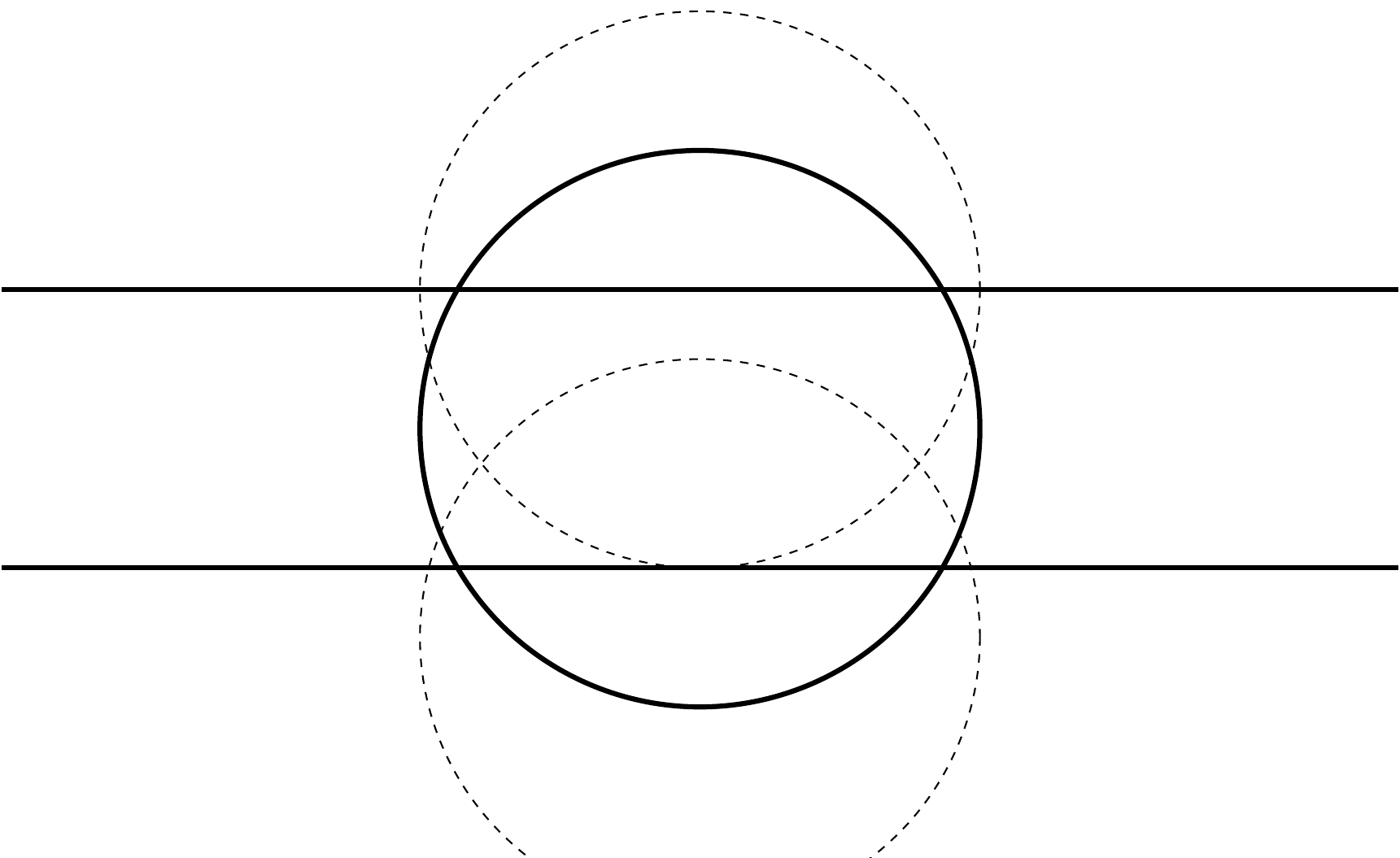}
	\caption{The ball in bold line maximizes the volume of the intersection with the slab $\Omega$.}
	\label{fig:pallastriscia}
\end{figure} 
In other words, this volume is minimal if the ball and $\Omega_m$ are concentric. We also observe that such a minimal value is given by
\[
\begin{split}
	|B_r\setminus\Omega_m|=|B_r\setminus \Omega|&=2\,\omega_{N-1}\,\int_{1}^r (r^2-z^2)^\frac{N-1}{2}\,dz\\
	&=2\,\omega_{N-1}\,r^{N}\,\int_{\arcsin\frac{1}{r}}^{\frac{\pi}{2}} \cos^N t\,dt:=r^N\,\varphi_N(r).
\end{split}
\]
In conclusion, we get that 
\begin{equation}\label{raggio-polya-szego}
r^*\ge r\,\left(\frac{\varphi_N(r)}{\omega_N}\right)^\frac{1}{N}=:r\,\Phi_N(r).
\end{equation}
From \eqref{1}, \eqref{2}, the monotonicity of the capacity with respect to the set inclusion, the lower bound on $r^*$ and again the scaling relations for the capacity, we get
\begin{equation}
\label{mipiacelaphi}
\big(\Phi_N(r)\big)^{N-p}\,\mathrm{cap}_p(\overline{B_1};B_{2/\Phi_N(r)})\le \gamma\,\mathrm{cap}_p(\overline{B_1};B_{2}).
\end{equation}
This relation must be satisfied by every radius $r>1$, such that $\overline{B_r(x_0)}\setminus\Omega$ is $(p,\gamma)-$negligible. 
\par
By recalling \eqref{cappalla}, the previous inequality is equivalent to
\[
\big(\Phi_N(r)\big)^{N-p}\,\left(1-\left(\dfrac{1}{2}\right)^\frac{N-p}{p-1}\right)^{p-1}\le \gamma\,\left(1-\left(\dfrac{\Phi_N(r)}{2}\right)^\frac{N-p}{p-1}\right)^{p-1}.
\]
With simple algebraic manipulations, we get that for every admissible radius $r$, we must have
\[
\Phi_N(r)\le \left(\frac{2^\frac{N-p}{p-1}\,\gamma^\frac{1}{p-1}}{2^\frac{N-p}{p-1}-1+\gamma^\frac{1}{p-1}}\right)^\frac{p-1}{N-p}.
\]
Observe that the right-hand side is strictly smaller than $1$. Moreover, by construction the function $r\mapsto \Phi_N(r)$ is continuous monotone increasing, with 
\begin{equation}\label{limit_beahviour_Phi}
\lim_{r\searrow 1} \Phi_N(r)=0\qquad \text{and}\qquad \lim_{r\nearrow +\infty} \Phi_N(r)=1.
\end{equation}
This implies that there exists a finite radius $r_\gamma>1$ such that 
\[
\Phi_N(r_\gamma)=\left(\frac{2^\frac{N-p}{p-1}\,\gamma^\frac{1}{p-1}}{2^\frac{N-p}{p-1}-1+\gamma^\frac{1}{p-1}}\right)^\frac{p-1}{N-p},
\]
and that every ball with radius $r>r_\gamma$ violates the previous conditions, i.e. it {\it is not} $(p,\gamma)-$negligibile. This finally proves that 
\[
R_{p,\gamma}(\Omega)\le r_\gamma<+\infty,
\]
as claimed.
\end{example}

\begin{example}[Degeneration for $\gamma\nearrow 1$]
\label{exa:poliquin}
We maintain the same notation as in Example \ref{exa:app_lower_bound_fail} and take again $\Omega = \mathbb{R}^{N-1} \times (-1,1)$. Since this set is bounded in the direction $\mathbf{e}_N$, we have $\lambda_p(\Omega)>0$.
We claim that 
\begin{equation}
	\label{raggiovaffa}
	\lim_{\gamma\nearrow 1} R_{p,\gamma}(\Omega)=+\infty.
\end{equation}
This proves that an upper bound of the type
\[
\lambda_{p}(\Omega)\le \widetilde{C}_{N,p,\gamma}\,\left(\frac{1}{R_{p,\gamma}(\Omega)}\right)^p,
\] 
with $\widetilde{C}_{N,p,\gamma}$ staying bounded for $\gamma$ converging to $1$, {\it can not be true}. 
\par
In order to show \eqref{raggiovaffa}, we first recall that the function
\[
\gamma \mapsto R_{p,\gamma}(\Omega), \qquad \text{with}\ \gamma \in [0, 1),
\]
is monotone non-decreasing. 
Thus, the limit in \eqref{raggiovaffa} exists. For every $r > 1$, we set 
\[
\gamma_r = \dfrac{\mathrm{cap}_p(\overline{B_r}\setminus \Omega; B_{2r})}{\mathrm{cap}_p(\overline{B_r};B_{2r})}<1,
\]
thus $\overline{B_r}\setminus \Omega$ is $\gamma_r-$negligible, obviously. Accordingly, we get from \eqref{mipiacelaphi}
\[
\big(\Phi_N(r)\big)^{N-p}\,\frac{\mathrm{cap}_p(\overline{B_1};B_{2/\Phi_N(r)})}{\mathrm{cap}_p(\overline{B_1};B_{2})}\le \gamma_r.
\]
By recalling \eqref{limit_beahviour_Phi}, from the previous inequality we get
\[
\lim_{r \nearrow +\infty} \gamma_r = 1.
\]
In particular, by monotonicity we have
\[
\lim_{\gamma\nearrow 1} R_{p,\gamma}(\Omega)=	\lim_{r \nearrow +\infty} R_{p, \gamma_r}(\Omega).
\]
On the other hand, since the set $\overline{B_{r}} \setminus \Omega$ is $\gamma_r-$negligible, we must have 
\[
R_{p,\gamma_r}(\Omega)\ge r.
\]	
By joining the last two facts, we finally obtain \eqref{raggiovaffa}.
\end{example}

\end{document}